\documentclass[11pt,reqno]{amsart}

\usepackage[usenames]{color}
\usepackage{graphicx}
\usepackage{amscd}
\usepackage[colorlinks=true,
linkcolor=webgreen,
filecolor=webbrown,
citecolor=webgreen]{hyperref}

\definecolor{webgreen}{rgb}{0,.5,0}
\definecolor{webbrown}{rgb}{.6,0,0}
\usepackage{amsthm}
\usepackage{fullpage}
\usepackage{enumerate}
\usepackage{epstopdf}
\usepackage{graphics,amsmath,amssymb}

\usepackage{amsfonts}
\usepackage{latexsym}
\usepackage{url}

\newcommand{\seqnum}[1]{\href{http://oeis.org/#1}{\underline{#1}}}

\theoremstyle{plain}
\newtheorem{theorem}{Theorem}

\newtheorem{lemma}[theorem]{Lemma}
\newtheorem{proposition}[theorem]{Proposition}

\theoremstyle{remark}

\def\adots{
  \mathinner{\mkern1mu\raise1pt\hbox{.}\mkern2mu\raise4pt\hbox{.}
  \mkern2mu\raise7pt\vbox{\kern7pt\hbox{.}}\mkern1mu}}

\begin{document}

\begin{center}

 \title[Fibonacci identities from Jordan Identities]
 {Fibonacci identities from Jordan Identities}
 \author{Santiago Alzate}
 \address{Instituto de Matem\'aticas\\
		Universidad de Antioquia\\
Medell\'in\\
Colombia.}
\email{santiago.alzate9@udea.edu.co}
 \author{Oscar Correa}
 \address{Instituto de Matem\'aticas\\
		Universidad de Antioquia\\
Medell\'in\\
Colombia.}
\email{oscar.correa@udea.edu.co}
  \author{Rigoberto Fl\'orez}
 \address{Department of Mathematical Sciences\\
		The Citadel\\
		Charleston, SC \\
		U.S.A.}
  \email{rigo.florez@citadel.edu}
 \end{center}



\date{\today}

\begin{abstract} 
In this paper, we connect two well established theories, the Fibonacci numbers and the Jordan algebras. We give a series of matrices, from literature, 
used to obtain recurrence relations of second-order and polynomial sequences. We also give some identities known in special Jordan Algebras. 
The matrices play a bridge role between  both theories. The  mentioned matrices connect both areas of mathematics, special Jordan algebras  
and recurrence relations, to obtain new identities and classic identities in Fibonacci numbers, Lucas numbers, Pell numbers, binomial transform, 
tribonacci numbers, and polynomial sequences among others.  The list of identities in this paper contains just a few examples of many that the 
reader can find using this technique.
\end{abstract}

\thanks{This work was partially supported by The Citadel foundation.}

\maketitle


\section{Introduction}
Many authors have used power of matrices to study recurrence relations. In 1981 Gould \cite{Gould} wrote a historical paper about the origins of using 
matrices in research with the Fibonacci sequence. Gould's paper has a bibliography with 45 items. Since then many papers have appeared using 
this technique.  

The study of the Fibonacci sequence and its identities became more visible when in 1963 Hoggatt and Brousseau founded the Fibonacci 
Quarterly journal. By the same time researchers in another area of mathematics  were working actively finding identities in Jordan algebras  
---our interest here---  (see for example, \cite{Glennie, Hall,JacobsonMacDonald, Jacobson, JacobsonPaige}). These  two areas of mathematics may have  
several topics in common. Therefore, the main objective of this paper, through examples, is to show some connections between both,  the recursive  
sequences  and the special Jordan algebra identities. We are wondering if the experts in Jordan algebras can find a deeper  connection.  
There are still many things, on how this connection works, that we would like to understand better. For example, we believe there is a direct relationship 
between the power associativity in Jordan identities and the arguments of the Fibonacci recurrence.

In this paper, we use matrices to bridge recurrence relations identities with special Jordan algebras identities.  We take a collection of matrices associated 
to sequences (Fibonacci sequences, Lucas sequences, and matrices associated to other recursive identities) from the literature; we also take a collection of  
special Jordan algebras identities, from the literature, to obtain identities in numerical sequences. 

Using identities from abstract algebra we can obtain more complex, general, and sophisticated numerical identities. For example, we give classic identities, 
new identities, and very complex identities in Fibonacci identities, Lucas identities, Pell identities, and many others. 

Williams \cite{Williams} and Mc Laughlin \cite{McLaughlin} give simple forms to construct sequences from  $2 \times 2$ matrices. Here we use 
the technique given in  \cite{McLaughlin} and the special Jordan algebra identities to show a new form to construct identities for recursive 
relations of order two. 

\section{Some Previous Results and Motivation }\label{Fibonacci_Identities_from_Jordan_Identities}

In this section, we give a series of matrices, from literature, used to obtain recurrence relations of second-order and polynomial sequences. 
Most of these matrices can be found in \cite{Wani,Gould,MelhamShannon,koshy,sloane}. In Section \ref{GeneralCase}, there is a more  
general form for powers of matrices associated to recurrence relations of order two. 

Our aim is to use  matrices to connect the special Jordan algebra identities with the recurrence relations to obtain new identities
associated to numerical sequences or polynomial sequences. 

\subsection{Fibonacci Matrices and generalized Fibonacci matrices} From \eqref{FibonacciGenCase1} we obtain these sequences:  the matrix   
$\mathcal{F}_{1}^{n}$  is the matrix associated to Fibonacci sequence. The matrix $\mathcal{G}_{2}^{n}$ gives rise to \emph{Jacobsthal numbers}  
$a_{n} = a_{n-1} + 2a_{n-2}$, with  $a_{0} = 0, a_{1} = 1$ (\seqnum{A001045}).  From \cite{Gould} we have the general case, the matrix 
$\mathcal{G}_{b}^{n}$ gives rise to 
\begin{equation}\label{generalizedFibo}
g_{n} = g_{n-1} + b \cdot g_{n-2}, \quad \text{ with } \quad g_{0} = 0, \quad g_{1} = 1, \quad \text{ where } b \in \mathbb{Z}_{>0}.
\end{equation}
We now give sequences associated with some values of  $b$.   
From \eqref{FibonacciGenCase1} with $b=1$ we have $\mathcal{F}_{1}^n$ the Fibonacci sequence; the equation \eqref{FibonacciGenCase1} 
with $b=2$ gives the Jacobsthal numbers $J_n:=g_n$ see \seqnum{A001045}; the equation   \eqref{FibonacciGenCase1} with $b= 3$ 
gives  \seqnum{A006130}; the equation \eqref{FibonacciGenCase1} with $b=4$ gives  \seqnum{A006131};  the equation  
\eqref{FibonacciGenCase1} with $b=5$ gives \seqnum{A015440}; and the equation \eqref{FibonacciGenCase1} with $b= 6$ 
gives \seqnum{A015441}. We summarize these results in \eqref{FibonacciGenCasePowers1}.

\begin{equation}\label{FibonacciGenCase1}
\mathcal{F}_{1} :=
 \left[
	\begin{array}{ll}
		1 & 1 \\
		1 & 0 \\
	\end{array}
\right];
\hspace{.5cm}
\mathcal{G}_{2} :=
 \left[
	\begin{array}{ll}
		1 & 2 \\
		1 & 0 \\
	\end{array}
\right];
\hspace{.5cm}
\mathcal{G}_{b} :=
 \left[
	\begin{array}{ll}
		1 & b \\
		1 & 0 \\
	\end{array}
\right]. 
\end{equation}
The powers of these matrices are

\begin{equation}\label{FibonacciGenCasePowers1}
\mathcal{F}_{1}^{n} =
 \left[
\begin{array}{ll}
 F_{n+1} & F_{n} \\
 F_{n} & F_{n-1}\\
\end{array}
\right]; \quad 
\mathcal{G}_{2}^{n} =
\left[
\begin{array}{ll}
 J_{2 n-1} & J_{2 n} \\
 J_{2 n} & J_{2 n+1} \\
\end{array}
\right]; \quad 
\mathcal{G}_{b}^{n} =
 \left[
	\begin{array}{ll}
		g_{n+1} & b g_{n}\\
		g_{n+1} & b g_{n-1} \\
	\end{array}
\right].
\end{equation}

The powers of the matrix $\mathcal{L}$ give rise to a matrix where the entries are Lucas numbers and Fibonacci numbers \cite{Johnson}.  

\begin{equation}\label{FiboLucasMatricesM} 
\hspace{.5cm}
\mathcal{L}:= (1/2)
\left[
	\begin{array}{ll}
		1 & 5 \\
		1 & 1 \\
	\end{array}
\right]; 
\hspace{.5cm}
\mathcal{L}^n= (1/2)
\left[
	\begin{array}{ll}
		L_{n} & 5 	F_{n} \\
		F_{n} & L_{n} \\
	\end{array}
\right].
\end{equation}
 
The \emph{generalized Fibonacci numbers} are defined as $w_{n}=p w_{n-1}-q w_{n-2}$, where $w_{0}=0$, and 
$w_{1}=1$ for $p$ and $q$ in $\mathbb{Z}_{\ge 0}$.  This recurrence relation is represented by the power of the matrix 
$\mathcal{W}$  in \eqref{FibonacciGenCase3} (see \cite{Wani,Gould,MelhamShannon}). Particular cases of this sequence 
are in  \seqnum{A015518} and \seqnum{A006190}.  

\begin{equation}\label{FibonacciGenCase3}
\mathcal{W}:=
\left[
	\begin{array}{ll}
		p & -q \\
		1 & 0 \\
	\end{array}
\right];
\hspace{1cm}
\mathcal{W}^n=
\left[
	\begin{array}{ll}
		w_{n+1} & -qw_{n} \\
		w_{n} & -qw_{n-1} \\
	\end{array}
\right].
\end{equation}

\subsection{Pell matrices and generalized Pell matrices.} The matrices in \eqref{FibonacciGenCase2} are obtained from particular cases of \eqref{FibonacciGenCase3}  
(see also \cite{Bicknell,Gould}). Using \eqref{FibonacciGenCase2} and power matrices we have that: $\mathcal{P}_{2}^{n}$ gives rise to \emph{Pell numbers}  
$p_{n}=2 p_{n-1}+p_{n-2}$, where $p_{0}=0$, $p_{1}=1$;  the matrix $\mathcal{P}_{3}^{n}$ gives rise to $b_{n}=3 b_{n-1}+b_{n-2}$, 
where $b_{0}=0$, $b_{1}=1$, and in general $\mathcal{P}_{b}^{n}$ gives rise to $c_{n}=b \cdot c_{n-1}+c_{n-2}$, where $c_{0}=0$, $c_{1}=1$.
Sequences associated with some values of  $b$; $b=2$ gives \seqnum{A000129}; $b= 3$ gives \seqnum{A006190}; $b=4$ gives \seqnum{A001076}; 
$b=5$ gives \seqnum{A052918}; and $b= 6$ gives \seqnum{A005668}. We summarize these results in \eqref{FibonacciGenCasePowers2}. 

\begin{equation}\label{FibonacciGenCase2}
\mathcal{P}_{2}:=
 \left[
	\begin{array}{ll}
		2 & 1 \\
		1 & 0 \\
	\end{array}
\right];
\hspace{.5cm}
\mathcal{P}_{3} :=
 \left[
	\begin{array}{ll}
		3 & 1 \\
		1 & 0 \\
	\end{array}
\right];
\hspace{.5cm}
\mathcal{P}_{b}:=
 \left[
	\begin{array}{ll}
		b & 1 \\
		1 & 0 \\
	\end{array}
\right].
\end{equation}
The powers of these matrices are

\begin{equation}\label{FibonacciGenCasePowers2}
\mathcal{P}_{2}^n=
\left[
	\begin{array}{ll}
		p_{n+1} & p_{n} \\
		p_{n}  & p_{n-1} \\
	\end{array}
\right];
\quad
\mathcal{P}_{3}^n=
\left[
	\begin{array}{ll}
		b_{n+1} & b_{n} \\
		b_{n} &  b_{n-1}\\
	\end{array}
\right];
\quad
\mathcal{P}_{b}^n=
\left[
	\begin{array}{ll}
		c_{n+1} & c_{n} \\
		c_{n} &  c_{n-1}\\
	\end{array}
\right].
\end{equation}

\subsection{Fibonacci Polynomials}\label{Fibonacci_polynomials} The following matrices that give rise to Fibonacci polynomials can be found in \cite{koshy}.
\begin{equation}\label{FiboPolyMatricesT3} 
\mathcal{Q}(x) :=
 \left[
	\begin{array}{ll}
		x & 1 \\
		1 & 0 \\
	\end{array}
\right];
\hspace{.5cm}
\mathcal{Q}^n(x) =
 \left[
	\begin{array}{ll}
		F_{n+1}(x) & F_{n}(x)  \\
		F_{n}(x)  &  F_{n-1}(x)\\
	\end{array}
\right].
\end{equation}

\subsection{Special Jordan Algebra background}\label{MotivationSection} 
In this section, we give the background of special Jordan algebras and three identities needed to show the examples required for this motivation section.   
The identities in Lemma \ref{JordanIdentitiesM} are part of Lemma \ref{JordanIdentities} on page \pageref{JordanIdentities}.
Part of the discussion here and some notation can be found  in  \cite{JacobsonMacDonald, Jacobson, JacobsonPaige}. 

A Jordan algebra $\mathcal{A}$ is a non-associative algebra over a field not of characteristic $2$ whose multiplication satisfies that  $a\cdot b=b \cdot a$  
(commutative law) and  $(a^2 \cdot b)\cdot a=a^2\cdot (b \cdot a)$ (Jordan identity). Let $(\mathcal{A}, +, \times, *)$ be the vector space of all $n\times n$   
matrices over $\mathbb{R}$,  where $+$, $\times$, and $*$ are the matrix addition, matrix product, and the scalar product, respectively. For simplicity,   
we use $ab$ instead of $a \times b$. 
(In this paper $n=2$.)  The vector space  $\mathcal{A}$ gives rise to the \emph{special Jordan algebra}  $\mathcal{A}^{+}=(\mathcal{A}, +, \cdot, *)$, where the 
\emph{Jordan product} (denoted by $\cdot$) is defined as $a\cdot b=(ab+ba)/2$.  We use $ \{a,b,c \} $ to denote this ternary operation 
 \begin{equation}\label{TernaryOpEquiM} 
 \{a,b,c \} =(1/2)\left[(a b) c+ (c b) a\right].
\end{equation} 

\begin{lemma}[\cite{JacobsonMacDonald,JacobsonPaige}] \label{JordanIdentitiesM}  Let  $\mathcal{A}$ be a special Jordan algebra with the ternary  
operation $\{\cdot,\cdot,\cdot\}$. If $a,b,c \in \mathcal{A}$, where  $a\cdot b$ is the Jordan product, then these identities hold
 
\begin{enumerate}
\item \label{JordanIdentityM:1} $\{ a^{n}, a^{m}, b^{n} \}   =  a^{(m+n)} \cdot b^n $, 
\item \label{JordanIdentityM:2}  $\{ a^{l} ,\{a^{m},b,a^{m}\},a^{l} \}  = \{  a^{m+l},b,a^{m+l}\} $, 
\item \label{JordanIdentityM:3} $\{ a^{n}, b, a^{n} \} \cdot  c  = 2 \{  a^{n},( a^{n} \cdot   b ), c \} - \{ a^{2n}, b, c \} $, 
\item \label{JordanIdentityM:9} $\{ a^{m}, b, a^{n} \} \cdot a^{l}  = \{  a^{m},( b\cdot a^{l}),a^{n}\} $. 
\end{enumerate}
\end{lemma}

\section{Examples of applicability of the Jordan identities in numerical sequences}\label{Examples_applicability_Jordan identities}
In this section, we give some a few examples on how to apply identities from special Jordan algebras to obtain new identities of order two recurrences relations. 
For example, we show some new and old identities in Fibonacci numbers, generalized Fibonacci numbers, Lucas numbers, Pell numbers, and combinations of some of them.

\subsection{Example} As a first example we show an application of Lemma \ref{JordanIdentitiesM} Part \eqref{JordanIdentityM:1} to  $\mathcal{F}_{1}$ in 
\eqref{FibonacciGenCase1}.  In this example, we use the Jordan identity to prove Identity VI in \cite{Basin} (more general). Thus, we prove that $F_{2n+1}=F_{n+1}^2+F_{n}^2$.  
Letting $a=\mathcal{F}_{1}$ and $b=I_{2}$ (the 2-by-2 identity matrix) in Lemma \ref{JordanIdentitiesM} Part \eqref{JordanIdentityM:1} we obtain that 

\[
\begin{array}{cccc}
	 \left\{a^{m},a^{n},b\right\}&=&a^{(m+n)} \cdot b,\\
	\{\mathcal{F}_{1}^{m},\mathcal{F}_{1}^{n},I_{2}\}&=&\mathcal{F}_{1}^{m+n}\cdot I_2.\\
\end{array}
\]

This and \eqref{FibonacciGenCasePowers1} imply that 
\begin{equation}\label{example1Par1}
\Bigg{ \{ }
      \left[
	\begin{array}{ll}
		F_{m+1} & F_{m} \\
 		F_{m} & F_{m-1}\\
	\end{array}
	\right],
	\left[
	\begin{array}{ll}
		F_{n+1} & F_{n} \\
 		F_{n} & F_{n-1}\\
	\end{array}
	\right],
	\left[
	\begin{array}{ll}
		1 & 0\\
		0 & 1 \\
	\end{array}
	\right]
 \Bigg{ \} }
= \left[
	\begin{array}{ll}
		F_{m+n+1} & F_{m+n} \\
 		F_{m+n} & F_{m+n-1}\\
	\end{array}
	\right].
\end{equation}
Applying \eqref{TernaryOpEquiM} to the left side of this equality  and simplifying we have the identity  

\[
 \left[
\begin{array}{ll}
 F_m F_n+F_{m+1} F_{n+1} &  (F_n L_m+F_m L_n)/2\\
 (F_n L_m+F_m L_n)/2& F_{m-1} F_{n-1}+F_m F_n \\
\end{array}
\right]
= 
\left[
	\begin{array}{ll}
		F_{m+n+1} & F_{m+n} \\
 		F_{m+n} & F_{m+n-1}\\
	\end{array}
\right].
\]
Taking $m=n+1$ and simplifying we obtain the desired identity.  

\subsection{Example} We now  give a second example on the application of Lemma \ref{JordanIdentitiesM} Part \eqref{JordanIdentity:1} to  $\mathcal{F}_{1}$      
in \eqref{FibonacciGenCasePowers1} and $\mathcal{L}$ in \eqref{FiboLucasMatricesM}. Thus, letting $a=\mathcal{F}_{1}$ and $b=\mathcal{L}$ in   
Lemma \ref{JordanIdentitiesM} Part \eqref{JordanIdentityM:1} we obtain that 
\[
\begin{array}{cccc}
	 \left\{a^{n},a^{m},b^{n}\right\}&=&a^{(m+n)} \cdot b^{n},\\
	\{\mathcal{F}_{1}^{n},\mathcal{F}_{1}^{m},\mathcal{L}^{n}\}&=&\mathcal{F}_{1}^{m+n}\cdot \mathcal{L}^{n}.\\
\end{array}
\]

This, \eqref{FibonacciGenCasePowers1}, and \eqref{FiboLucasMatricesM} imply that 
\begin{multline}\label{example2Par2}
\Bigg{ \{ }
      \left[
	\begin{array}{ll}
		F_{n+1} & F_{n} \\
 		F_{n} & F_{n-1}\\
	\end{array}
	\right],
	\left[
	\begin{array}{ll}
		F_{m+1} & F_{m} \\
 		F_{m} & F_{m-1}\\
	\end{array}
	\right],
	\left[
	\begin{array}{ll}
		L_{n}/2 & 5 F_{n}/2 \\
		F_{n}/2 & L_{n}/2 \\
	\end{array}
	\right]
 \Bigg{ \} }=\\
 \left[
	\begin{array}{ll}
		F_{m+n+1} & F_{m+n} \\
 		F_{m+n} & F_{m+n-1}\\
	\end{array}
	\right]\cdot
	\left[
	\begin{array}{ll}
		L_{n}/2 & 5 F_{n}/2 \\
		F_{n}/2 & L_{n}/2 \\
	\end{array}
	\right]
=\\
\left[
\begin{array}{ll}
(L_n F_{m+n+1}+3 F_n F_{m+n} )/2 & (2 L_n F_{m+n}+5 F_n L_{m+n} )/4 \\
  (2 L_n F_{m+n}+F_n L_{m+n} )/4 & (L_n F_{m+n-1}+3 F_n F_{m+n} )/4 \\
\end{array}
\right].
\end{multline}

Applying \eqref{TernaryOpEquiM} and simplifying we have that the left side (top) of this last equation is equal to  
\begin{multline*}
 \left[
\begin{array}{ll}
  L_n (F_m F_n+F_{m+1} F_{n+1} )/2& 
 L_n (L_{m} F_n+F_m L_{n} )/4   \\
L_n (L_m  F_n+F_m L_n  )/4 & 
  L_n (F_{m-1} F_{n-1}+F_m F_n )/2 \\
\end{array}
\right]+\\
 \left[
\begin{array}{ll}
  3 F_n (F_{m-1} F_n+F_m F_{n+1} ) /2 & 
 5 F_n (F_{m-1} F_{n-1}+2 F_m F_n+F_{m+1} F_{n+1} ) /4 \\
  F_n (F_{m-1} F_{n-1}+2 F_m F_n+F_{m+1} F_{n+1} ) /4 & 
   3 F_n (F_m F_{n-1}+F_{m+1} F_n ) /2 \\
\end{array}
\right].
\end{multline*}

Since the entries of the sum of these last matrices are equal to the entries of the right side matrix (bottom) of  \eqref{example2Par2}, 
after doing some simplifications, we obtain these four identities. 

\[L_n F_{m+n+1}+3 F_n F_{m+n}= L_n (F_m F_n+F_{m+1} F_{n+1} )+3 F_n (F_{m-1} F_n+F_m F_{n+1} ).  \]
\[  2 L_n F_{m+n}+5 F_n L_{m+n}=L_n (L_{m} F_n+F_m L_{n} )+5 F_n (F_{m-1} F_{n-1}+2 F_m F_n+F_{m+1} F_{n+1} ).\]
\[ 2 L_n F_{m+n}+F_n L_{m+n}  = L_n (L_m  F_n+F_m L_n  ) +F_n (F_{m-1} F_{n-1}+2 F_m F_n+F_{m+1} F_{n+1} ) . \]
\[L_n F_{m+n-1}+3 F_n F_{m+n}= 2L_n (F_{m-1} F_{n-1}+F_m F_n )+6 F_n (F_m F_{n-1}+F_{m+1} F_n ) .\]

\subsection{Example} In this example, we apply special Jordan identities to Fibonacci polynomials. In this case, we use 
Lemma \ref{JordanIdentitiesM} Part \eqref{JordanIdentityM:3} with \eqref{FiboPolyMatricesT3}. We take $a^n=\mathcal{Q}^n(x)$,  
$b=\mathcal{Q}^m(x)$ and  $c=I_2$. So, 
\[
\begin{array}{rcl}
	 \{ a^{n}, b, a^{n} \} \cdot  c  &=& 2 \{  a^{n},( a^{n} \cdot   b ), c \} - \{ a^{2n}, b, c \}. \\
	  \{ \mathcal{Q}^n(x), \mathcal{Q}^m(x), \mathcal{Q}^n(x) \} \cdot  I_2  
	  &=& 2 \{  \mathcal{Q}^n(x),( \mathcal{Q}^n(x) \cdot   \mathcal{Q}^m(x) ), I_2 \} - \{ \mathcal{Q}^{2n}(x), \mathcal{Q}^m(x), I_2 \}.
\end{array}
\]
These give rise to the following identities. For simplicity of the identities we set $f_t =F_{t}(x)$ and $l_t =L_{t}(x)$ (Lucas polynomial) for every $t>0$. (For more identities in Fibonacci polynomials see \cite{Florez}.)

\[
f_{m-1} f_n^2+f_{n+1} (2 f_m f_n+f_{m+1} f_{n+1})=
f_{m-1} f_n^2+2f_mf_n f_{n+1}+f_{m+1} (f_n^2+2 f_{n+1}^2-f_{2 n+1}).
\]
\begin{multline*} 
f_n (f_{m-1} f_{n-1}+f_{m+1} f_{n+1})+f_m (f_n^2+f_{n-1} f_{n+1})=\\
 f_m (4 f_n^2+l_n^2-l_{2 n})/2+f_n (f_{m-1} f_{n-1}+f_{m+1} f_{n+1}).
\end{multline*}

\subsection{Example} In this example, we apply special Jordan identities combining Fibonacci numbers and Lucas numbers with a matrix having a variable.  
In this case, we use   Lemma \ref{JordanIdentitiesM} Part \eqref{JordanIdentityM:1} with $n=m$ and $a^n=\mathcal{L}^{n}$ and 
$b=\left[\begin{array}{ll}
 x& 1 \\
 1 & 0 \\
\end{array}
\right]$.
\[5 x F_n^2+6 L_n F_n+x L_n^2=6 F_{2 n}+2x L_{2 n},\]
\[5 F_n^2+5 x L_n F_n+L_n^2=5 x F_{2 n}+2 L_{2 n}.\]

\section{Recursive Relations from $2\times 2$ matrices} \label{GeneralCase}
This section is based on the results found by Mc Laughlin \cite{McLaughlin}. We now give a summary of the results from \cite{McLaughlin} that we are going to use here in this paper. 

Let $T:=a+d$ and $D:=ad-bc$ be the trace and the determinant of $A$,  where 
\[
\mathcal{A}=
\left[
	\begin{array}{ll}
		a & b  \\
		c & d\\
	\end{array}
\right].
\]
If $\alpha=(T+\sqrt{T^2-4D})/2$  and $\beta=(T-\sqrt{T^2-4D})/2$, then for $\alpha \ne \beta$, $I_2$ ---the $2\times 2$ identity matrix--- and 
\begin{equation}\label{McLaughlinBinet}
z_n:= \dfrac{\alpha^n-\beta^n}{\alpha-\beta} = \sum _{m=0}^{\lfloor \frac{n-1}{2}\rfloor }\binom{n}{2m+1} T^{n-2m-1} (T^2-4D)^m/2^{n-1},
\end{equation} 
this holds 

\begin{equation}\label{McLaughlinSeq}
\mathcal{A}^n=z_n\mathcal{A}-z_{n-1}D I_2.
\end{equation}

\begin{theorem}[\cite{McLaughlin}]\label{McLaughlinThm} If  
\[y_n=\sum _{i=0}^{\lfloor \frac{n}{2}\rfloor }\binom{n-i}{i} T^{n-2 i} (-D)^i, \]
then 
\[ A^n= \left[
\begin{array}{cc}
 y_n-dy_{n-1} & by_{n-1} \\
 c y_{n-1} & y_{n}-a y_{n-1} \\
\end{array}
\right]. 
\]
\end{theorem}

We have observed that if $A:=\{ \{1,1\},\{1,0\}\}$ and 
$$\overline{Z}_n:= \dfrac{\alpha^n+\beta^n}{T} = \sum _{i=0}^{\lfloor \frac{n}{2}\rfloor }\binom{n}{2i} T^{n-2i} (T^2-4D)^i/2^{n-1},$$ 
then the Lucas sequence can be obtained by 
$A^{n-1}B=\overline{Z}_n A-\overline{Z}_{n-1}D I_2$, where 
$B:=A^2+I_2=
\left[
\begin{array}{cc}
 3 & 1 \\
 1 & 2 \\
\end{array}
\right].
$
\subsection{Matrices associated to $k$-th binomial transform of Fibonacci numbers.} \label{GeneralCaseMcLaughlinThm}

We now give some examples of matrices using the technique in Theorem \ref{McLaughlinThm} and \eqref{McLaughlinSeq}. The first entries of the matrices  
$\mathcal{T}^n_{k+1}$ given in   
\eqref{TMatrices} give rise to the $k$-th  binomial transform of  $F_{k+1}$ (see \cite{RugglesIII}).  For the particular case $\mathcal{T}^{n}_{2}$ gives rise to  
$\{F_{2n+1}\}$ and $\{F_{2n}\}$ see  \cite{koshy,Moore}.

In general, $\mathcal{T}^n_{k+1}$ gives rise to the sequences 
$$h_{n,k}(j)=\sum _{i=0}^n (-1)^{i-1+j}  \binom{n}{i} F_{i-j} (k+1)^{n-i}.$$ 
We summarize these results in \eqref{TMatricesPower}.  When $k$ varies for small values the sequences are in \cite{sloane}.  
For example, when $k=2$ we obtain the sequence  $d_{n}=5d_{n-1}-5d_{n-2}$, where  the initial conditions 
depend on the position in the matrix. For example, the sequence associated to the entry $(1,1)$ of 
$\mathcal{T}^{n}_{3}$ is  $d_{n,11}=5d_{n-1}-5d_{n-2}$, where $d_{0}=1$, $d_{1}=3$; 
the sequence associated to the entries $(1,2)$ or $(2,1)$ is $d_{n,12}=5d_{n-1}-5d_{n-2}$, where $d_{0}=1$, $d_{1}=5$;  
and the sequence associated to the entry  $(2,2)$ is $d_{n,22}=5d_{n-1}-5d_{n-2}$, where $d_{0}=2$, and $d_{1}=5$ (see \seqnum{A081567}, \seqnum{A030191}, and \seqnum{A020876}). 

\begin{equation}\label{TMatrices}
\mathcal{T}_{2}:=
\left[
	\begin{array}{ll}
		2 & 1 \\
		1 & 1 \\
	\end{array}
\right]; 
\quad
\hspace{.5cm}
\mathcal{T}_{k+1}:=
\left[
	\begin{array}{ll}
		k+1 & 1 \\
		1 & k \\
	\end{array}
\right].
\end{equation}
The powers of these matrices are

\begin{equation}\label{TMatricesPower}
\mathcal{T}^{n}_{2}=
\left[
	\begin{array}{ll}
		F_{2 n+1} & F_{2 n} \\
		F_{2 n} & F_{2 n-1} \\
	\end{array}
\right]; 
\hspace{.5cm}
\mathcal{T}^n_{k+1}=
\left[
	\begin{array}{ll}
		h_{n,k}(1) & h_{n-1,k}(0)\\
		h_{n-1,k}(0)  & h_{n,k}(-1) \\
	\end{array}
\right].
\end{equation}

\subsection{Other matrices}\label{Combinatorial_McLaughlin} The following matrices can be found in \cite{McLaughlin}.

\begin{equation}\label{Combinatorial_McLaughlin1} 
\mathcal{M}_1 :=
 \left[
	\begin{array}{ll}
		2 & 1 \\
		-1 & 0 \\
	\end{array}
\right]^n
=
 \left[
	\begin{array}{ll}
		n+1 & n  \\
		-n  &  -n+1\\
	\end{array}
\right].
\end{equation}

\begin{equation}\label{Combinatorial_McLaughlin2}  
\mathcal{M}_2 :=
 \left[
	\begin{array}{ll}
		3 & 1 \\
		-2 & 0 \\
	\end{array}
\right]^n
=
 \left[
	\begin{array}{ll}
		2^{n+1}-1 & 2^n-1  \\
		-2^{n+1}+2 & -2^n+2\\
	\end{array}
\right].
\end{equation}

\begin{equation}\label{Combinatorial_McLaughlin3} 
\mathcal{M}_3 :=
 \left[
	\begin{array}{ll}
		-2 & -1 \\
		1 & 1 \\
	\end{array}
\right]^n
=(-1)^{n}
 \left[
	\begin{array}{ll}
		F_{n+2}  &  F_{n}\\
		-F_n  &  -F_{n-2}\\
	\end{array}
\right].
\end{equation}

\subsection{Example} In this example, we apply special Jordan identities to  $k$-th binomial transform of Fibonacci numbers. In this case, we use 
Lemma \ref{JordanIdentitiesM} Part \eqref{JordanIdentityM:9}. We take $a=\mathcal{T}_{k+1}$ from \eqref{TMatricesPower}, $b=I_2$ and   
$c=\mathcal{T}_{k+1}$ from \eqref{TMatricesPower}. 
So, the entries (1,1) of all matrices give 
$$h_{n,k}^{3}(1) =h_{n,k}(0) h_{2n,k}(0) + h_{n,k}(1) h_{2n,k}(1) -  h_{n,k}^{2}(0) (2 h_{n,k}(1) + h_{n,k}(-1)).$$ 
This is equivalent to
\begin{multline*} 
 \left(\sum_{i=0}^{n}(-1)^{i}\binom{n}{i}F_{i-1}(k+1)^{n-i}\right)^{3}=\\
\sum_{i=0}^{n}(-1)^{i-1}\binom{n}{i}F_{i}(k+1)^{n-i}\sum_{i=0}^{2n}(-1)^{i-1}\binom{2n}{i}F_{i}(k+1)^{2n-i} +\\
\sum_{i=0}^{n}(-1)^{i}\binom{n}{i}F_{i-1}(k+1)^{n-i}\sum_{i=0}^{2n}(-1)^{i}\binom{2n}{i}F_{i-1}(k+1)^{2n-i}+\\ 
-  \left(\sum_{i=0}^{n}(-1)^{i-1}\binom{n}{i}F_{i}(k+1)^{n-i}\right)^{2}\left(\sum_{i=0}^{n}(-1)^{i}\binom{n}{i}(2F_{i-1}+F_{i+1})(k+1)^{n-i}\right).
\end{multline*} 

\subsection{Fibonacci-Lucas matrix.} 
The powers of the matrix $\mathcal{L}$ give rise to a matrix where the entries are Lucas numbers and Fibonacci numbers \cite{Johnson}.  

\begin{equation}\label{FiboLucasMatrices} 
\hspace{.5cm}
\mathcal{L}:= (1/2)
\left[
	\begin{array}{ll}
		1 & 5 \\
		1 & 1 \\
	\end{array}
\right];
\hspace{.5cm}
\mathcal{L}^n= (1/2)
\left[
	\begin{array}{ll}
		L_{n} & 5 	F_{n} \\
		F_{n} & L_{n} \\
	\end{array}
\right].
\end{equation}
 
The powers of matrix $\mathcal{S}_k$ give rise to the sequences $\{(k^2+1)^n\}$ and $\{k(k^2+1)^n\}$. 
 Since the matrix $\mathcal{S}_{k}^2$ in \eqref{SMatrices} is diagonalizable, it is easy to see that the matrices 
 $S^{2n}_{k}$ and $S^{2n+1}_{k}$ are correct. 

\begin{equation*}
\mathcal{S}_{k}:=
\left[
	\begin{array}{ll}
		1 & k \\
		k & -1 \\
	\end{array}
\right];
\hspace{1cm}
\mathcal{S}_{k}^{2}=
\left[
\begin{array}{ll}
 k^2+1 & 2 k \\
 2 k & k^2+1 \\
\end{array}
\right].
\end{equation*}

\begin{equation}\label{SMatrices} 
\hspace{.5cm}
\mathcal{S}_{k}^{2n}=
\left[
\begin{array}{ll}
 (k^2+1)^n & 0 \\
 0 & (k^2+1)^n \\
\end{array}
\right];
\hspace{1cm}
\mathcal{S}_{k}^{2n+1}=
\left[
\begin{array}{ll}
 (k^2+1)^n & k(k^2+1)^n \\
 k(k^2+1)^n & (k^2+1)^n \\
\end{array}
\right].
\end{equation}

\subsection{Tribonacci identities} In this section, we give matrices associated to third-order recurrence relations. 
For example, the matrix associated to the tribonacci sequence is denoted by $\mathcal{T}_{0,0,1}$, where the sequence generated by 
the powers of $\mathcal{T}_{0,0,1}$  is given by  $t_{n}=t_{n-1}+t_{n-2}+t_{n-3}$, where $t_{0}=0$, $t_{1}=0$, and $t_{2}=1$ \cite{Basu,Waddill}.   
The sequence generated by the powers of the matrix  $\mathcal{T}_{1,2,1}$ is $s_{n}=s_{n-1}+2s_{n-2}+s_{n-3}$, where  
$s_{0}=0$, $s_{1}=1$, and $s_{2}=1$ \cite{Waddill}.
The sequence generated the powers of the matrix  $\mathcal{T}_{r,s,t}$ is $u_{n}=ru_{n-1}+s u_{n-2}+t u_{n-3}$, where  
$u_{0}=0$, $u_{1}=1$, and $u_{2}=r$ \cite{Waddill}.
For matrices in \eqref{TribunacciMatricesT3} see \cite{koshy}. 

\begin{equation}\label{TribunacciMatricesT1} 
\mathcal{T}_{0,0,1} :=
 \left[
	\begin{array}{lll}
		1 & 1 & 1\\
		1 & 0  & 0\\
		0 & 1  & 0\\
	\end{array}
\right];
\hspace{1cm}
\mathcal{T}^n_{0,0,1} :=
 \left[
	\begin{array}{lll}
		t_{n+2} &t_{n}+t_{n+1} & t_{n+1}\\
		t_{n+1} & t_{n}+t_{n-1}  & t_{n}\\
		t_{n} &t_{n-1}+t_{n-2}  & t_{n-1}\\
	\end{array}
\right].
\end{equation}

\begin{equation}\label{TribunacciMatricesT2} 
\mathcal{T}_{1,2,1} :=
 \left[
	\begin{array}{ccc}
		1 & 2 & 1\\
		1 & 0  & 0\\
		0 & 1  & 0\\
	\end{array}
\right];
\hspace{1cm}
\mathcal{T}^{n}_{1,2,1} :=
 \left[
	\begin{array}{ccc}
		s_{n+1} & 2s_{n}+s_{n-1}& s_{n}\\
		s_{n} &  2s_{n-1}+s_{n-2} & s_{n-1}\\
		s_{n-1} & 2s_{n-2}+s_{n-3}  & s_{n-2}\\
	\end{array}
\right].
\end{equation}

\begin{equation}\label{TribunacciMatricesT4} 
\mathcal{T}_{r,s,t}:=
 \left[
	\begin{array}{ccc}
		r & s & t\\
		1 & 0  & 0\\
		0 & 1  & 0\\
	\end{array}
\right];
\hspace{1cm}
\mathcal{T}^{n}_{r,s,t} :=
 \left[
	\begin{array}{ccc}
		u_{n+1} & s u_{n}+t u_{n-1}& u_{n}\\
		u_{n} &  s u_{n-1}+tu_{n-2} & u_{n-1}\\
		u_{n-1} & su_{n-2}+tu_{n-3}  & u_{n-2}\\
	\end{array}
\right].
\end{equation}

\begin{equation}\label{TribunacciMatricesT3} 
\mathcal{T}_{F} :=
 \left[
	\begin{array}{ccc}
		0 & 0 & 1\\
		0 & 1  & 2\\
		1 & 1  & 1\\
	\end{array}
\right];
\hspace{.5cm}
\mathcal{T}^{n}_{F} :=
 \left[
	\begin{array}{ccc}
		F_{n-1}^{2} & F_{n-1} F_n & F_{n}^{2}\\
		2 F_{n-1}F_n  & F_{n-1}^{2}+F_{n+1} F_n  & 2 F_{n+1} F_n\\
		F_{n}^{2} & F_n F_{n+1}  & F_{n+1}^{2}\\
	\end{array}
\right].
\end{equation}

\subsection{Example} In this case, we use Lemma \ref{JordanIdentitiesM} Part \eqref{JordanIdentityM:1}. We take $a=\mathcal{T}_{0,0,1}$ from 
\eqref{TribunacciMatricesT1} and $b=I_3$.
\begin{enumerate}
\item $t_{m+n+2}=t_{m} t_{n+1}+t_{m+1} (t_{n}+t_{n+1})+t_{m+2} t_{n+2}$.

\item $2 t_{m+n+1}=t_{m-1} t_{n+1}+t_{m+2} t_{n+1}+t_{m} (2 t_{n}+t_{n+1})+t_{m+1} (t_{n-1}+t_{n}+t_{n+2})$.
\end{enumerate}

\section{Identities in Jordan Algebras} \label{JordanAlgebras}
In this section, we give a series of special Jordan algebra  identities from classic literature \cite{JacobsonMacDonald, Jacobson, JacobsonPaige}  (a few identities of many in the literature). 

\subsection{Special Jordan Algebra background} In this section, we complete the identities given in Subsection \ref{MotivationSection}. We recall that  the 
Jordan product is defined as $a\cdot b=(ab+ba)/2$ and that $ \{a,b,c \} $ denotes the ternary operation 
 \begin{equation}\label{TernaryOpEqui} 
 \{a,b,c \} =(1/2)\left[(a b) c+ (c b) a\right].
\end{equation} 

\begin{lemma}[\cite{JacobsonMacDonald,JacobsonPaige}] \label{JordanIdentities} Let  $\mathcal{A}$ be a special Jordan algebra with the ternary       
operation $\{\cdot,\cdot,\cdot\}$.  If $a,b,c \in \mathcal{A}$, where  $a\cdot b$ is the Jordan product, then these identities hold
 
\begin{enumerate}
\item \label{JordanIdentity:1} $\{ a^{n}, a^{m}, b^{n} \}   =  a^{(m+n)} \cdot b^n $. 
\item \label{JordanIdentity:2}  $\{ a^{l} ,\{a^{m},b,a^{m}\},a^{l} \}  = \{  a^{m+l},b,a^{m+l}\} $. 
\item \label{JordanIdentity:3} $\{ a^{n}, b, a^{n} \} \cdot  c  = 2 \{  a^{n},( a^{n} \cdot   b ), c \} - \{ a^{2n}, b, c \} $. 
\item \label{JordanIdentity:4} $2(\{ a^{nm}, b, c \} \cdot  a^{n} )  = \{  a^{n}, \{  a^{(mn-n)}, b, c \}, a^{n} \} + \{  a^{(mn+n)},  b,  c \}  $. 
\item \label{JordanIdentity:6} $2(\{a^{ n},b,a^{n}\} \cdot  a^{ n} )  = \{ a^{ n}, b \cdot a^{ n},  a^{ n} \} + \{ a^{ 2n}, b,a^{ n}  \}  $. 
\item \label{JordanIdentity:9} $\{ a^{m}, b, a^{n} \} \cdot a^{l}  = \{  a^{m},( b\cdot a^{l}),a^{n}\} $. 
\item \label{JordanIdentity:10} $\{ a^{m}, b, a^{m} \} \cdot a^{l}  = \{  a^{m+l},b,a^{m}\} $.   
\item \label{JordanIdentity:12} $\{ a^{l} , \{a^{m},b, a^{n} \} ,c \}= \{a^{(l+m)}, b\cdot a^{n}, c \} + \{ a^{(l+n)}, b\cdot a^{m}, c \} - \{a^{(l+m+n)}, b, c \}$. 
\item \label{JordanIdentity:13} $\{ a^{l} , \{  a^{m} , b,c \}, a^{n} \}  =  \{  a^{(l+m)}, b, c \}\cdot a^{n} + \{ a^{(m+n)},b, c \}\cdot a^{l} - \{  a^{(l+m+n)},  b,  c \}$. 
\end{enumerate}
\end{lemma}

\begin{lemma}[\cite{JacobsonPaige,Macdonald}] \label{JordanIdentitiesLemma2}  Let  $\mathcal{A}$ be a special Jordan algebra with the ternary    
operation $\{\cdot,\cdot,\cdot\}$. If $a,b,c \in \mathcal{A}$, where  $a\cdot b$ is the Jordan product, then these identities hold
 
\begin{enumerate}
\item  \label{JordanIdentityL2Part:1} $2\{a^{n},b,c \} \cdot a=\{a,\{a^{n-1},b,c\},a\}+\{a^{n+1},b,c\}$. 
\item  \label{JordanIdentityL2Part:2} $\{a^{n},\{a^{m},b,a^{m}\},c\}=2\{a^{n+m},(a^{m} \cdot b),c\}-\{a^{n+2m},b,c\}$. 
\end{enumerate}
\end{lemma}

\section{Proving classical Fibonacci identities using Jordan identities}\label{ClassicalFibonacciProvedJordanIdentities}

As an example, of the application of the Jordan algebras in numerical sequences, we give different proofs of some classic identities.   
The proofs in this section are obtained applying 
just one of Jordan identitites (Lemma \ref{JordanIdentities} Part \eqref{JordanIdentity:1}). Note it is one of the simpler Jordan identity, so this shows that Jordan identities
are also a great tool  to re-prove classical identities.  For example,  Identity Part \eqref{ClassicIdentitiesPart1} is the Lucas identity \cite{koshy, Vajda}, Identities Parts   
\eqref{ClassicIdentitiesPart2}, \eqref{ClassicIdentitiesPart3}, \eqref{ClassicIdentitiesPart4}, \eqref{ClassicIdentitiesPart7}, are in \cite{Vajda},   
Identities in Parts \eqref{ClassicIdentitiesPart5}, 
\eqref{ClassicIdentitiesPart8} are in \cite{Benjamin},  Identity in Part \eqref{ClassicIdentitiesPart6}  is in \cite{Ferns}, the Identities in   
Parts  \eqref{ClassicIdentitiesPart9} and \eqref{ClassicIdentitiesPart10} are applications of Part \eqref{ClassicIdentitiesPart8}.

\begin{proposition} \label{ClassicIdentities}  For $n\ge 1$, these hold.

\begin{enumerate}
\item \label{ClassicIdentitiesPart1}  $F_{n}^2+F_{n+1}^2 =   F_{2 n+1}$, 
\item \label{ClassicIdentitiesPart2}  $ F_{n} (F_{n-1}+F_{n+1})=  F_{2 n}$, 
\item \label{ClassicIdentitiesPart3}  $5 F_n^2+L_n^2=2L_{2 n}$, 
\item \label{ClassicIdentitiesPart4} $ F_n L_n=F_{2 n}$,  
\item \label{ClassicIdentitiesPart5} $F_m F_n+F_{m+1} F_{n+1}= F_{m+n+1}$,
\item \label{ClassicIdentitiesPart6} $ 5 F_m F_n+L_m L_n=2L_{m+n}$,
\item \label{ClassicIdentitiesPart7} $F_n L_m+F_m L_n=2F_{m+n}$,
\item \label{ClassicIdentitiesPart8} $F_m F_{n-1}+F_{m+1} F_n= F_{m+n}$,
\item \label{ClassicIdentitiesPart9} $F_m F_n+2 F_{m+1} F_{n+1}+F_{m+2} F_{n+2}=F_{m+n+1}+F_{m+n+3}$,
\item \label{ClassicIdentitiesPart10} $F_{n-1}^2+2 F_n^2+F_{n+1}^2=F_{2 n-1}+F_{2 n+1}$.
 \end{enumerate}
 \end{proposition}
 
 \begin{proof} The proofs of all parts of this proposition follow from Lemma \ref{JordanIdentities} Part \eqref{JordanIdentity:1}. Therefore, here  
 we indicate the matrices used for 
 $a^n$, $a^m$, and $b$. For the  proof of Parts \eqref{ClassicIdentitiesPart1} and \eqref{ClassicIdentitiesPart2}, we use $a^n=\mathcal{F}^n_1$, 
 $a^m=\mathcal{F}^n_1$ from \eqref{FibonacciGenCasePowers1} and 
 $b=
 \left[\begin{array}{ll}
 1 & 0 \\
 0 & 1 \\
\end{array}
\right]$. 

The proof of Parts \eqref{ClassicIdentitiesPart3} and \eqref{ClassicIdentitiesPart4}, uses $n=m$, 
$a^n=\mathcal{L}^n$ from   \eqref{FiboLucasMatrices} with  $b=\left[\begin{array}{ll}
 1 & 0 \\
 0 & 1 \\
\end{array}
\right]$. 

 The proof of Part \eqref{ClassicIdentitiesPart5}, uses $a^n=\mathcal{F}^n_1$, $a^m=\mathcal{F}^m_1$ from    \eqref{FibonacciGenCasePowers1} with   $b=\left[\begin{array}{ll}
 1 & 0 \\
 0 & 0 \\
\end{array}
\right]$.  
 
 The proof of Parts \eqref{ClassicIdentitiesPart6} and  \eqref{ClassicIdentitiesPart7}, uses $a^n=\mathcal{L}^n$, $a^m=\mathcal{L}^m$ from    \eqref{FiboLucasMatrices}   
 with  
 $b=\left[\begin{array}{ll}
 0& 1 \\
 0 & 0 \\
\end{array}
\right]$. 

 The proof of Part \eqref{ClassicIdentitiesPart8}, uses $a^n=\mathcal{F}^{n}_1$, $a^m=\mathcal{F}^{m}_1$ from    \eqref{FibonacciGenCasePowers1} with   $b=\left[\begin{array}{ll}
 0& 0 \\
 0 & 1 \\
\end{array}
\right]$.  

 The proof of Part \eqref{ClassicIdentitiesPart9}, uses $a^n=\mathcal{F}^{n+1}_1$, $a^m=\mathcal{F}^{m+1}_1$ from    \eqref{FibonacciGenCasePowers1} with   $b=\left[\begin{array}{ll}
 0& 1 \\
 0 & 0 \\
\end{array}
\right]$.  

The proof of Part \eqref{ClassicIdentitiesPart10}, uses $n=m$, $a^n=\mathcal{F}^{n}_1$ from    \eqref{FibonacciGenCasePowers1} with   $b=\left[\begin{array}{ll}
0& 1 \\
 1 & 0 \\
\end{array}
\right]$. 
 \end{proof}
 
 \section{Recursive relations identities from Jordan identities}\label{Recursive_Identities_from_Jordan_Identities}
 
Using the mentioned matrices in Sections \ref{Fibonacci_Identities_from_Jordan_Identities} and \ref{GeneralCase}, and the identities in
Section \ref{JordanAlgebras}, we connect both areas of   mathematics, special Jordan algebras and recurrence relations. Here we give a collection  
of identities of Fibonacci numbers, Lucas numbers, Pell numbers, and the binomial transform. This list is not complete, these identities are actually   
a few examples of many that  the reader can find using this technique. Since the main objective of this paper is to show the path between special   
Jordan algebras and the recurrences relations, the identities are simplified but not too deep. 

\subsection{Fibonacci and other identities from Jordan identities}\label{Fibonacci_Identities_from_Jordan_IdentitiesExam}

The proofs of the following theorems are straightforward applications of the identities given in Lemmas \ref{JordanIdentities}  and  \ref{JordanIdentitiesLemma2}.
and the matrices that are given in Sections \ref{Fibonacci_Identities_from_Jordan_Identities}. The proofs are made following the technique  
used in Section \ref{Examples_applicability_Jordan identities}.

 \begin{proposition}\label{SantiagoTheorem7} If $F_{n}$ is a Fibonacci number and $L_{n}$ is a Lucas number, then these identities hold
\begin{enumerate}
\item  \label{FibonaIdentitySat:OSC1}
$F_{m} (F_{2 n} F_{n+1}+F_{n} F_{2n+1})= F_{2n+1}(  F_{m+n+1} -F_{m+1} F_{n+1}) - F_{2 n}( F_{m-1} F_{n}- F_{m+n}), $
  
\item  \label{FibonaIdentitySat:OSC2}  
$5 F_{n}^2= F_{n+2}( 2F_{n+3}  -3 F_{n-1})-F_{n+1}(  6 F_{n}+ F_{n+1})$,
 
\item  \label{FibonaIdentitySat:OSC3}  
$ F_{n+1} ( 2 F_{n+2} -F_{n+1})  =F_{n}F_{n+3}    +F_{n-1}F_{n+2}$,

\item  \label{FibonaIdentitySat:OSC4}   
$5 (F_{n+1}^2+F_{n}^2)=4 F_{2n}+5 F_{2n+1}-4F_{n}(F_{n-1} + F_{n+1})$,

\item  \label{FibonaIdentitySat:OSC5}  
$ 11F_{n+1}^2=  13 F_{2n}+6 F_{2n-1}+11 F_{2n+1} -6F_{n-1}^2 -17F_{n}^2-13F_{n}(F_{n-1}+ F_{n+1})$,
 
\item  \label{FibonaIdentitySat:OSC6}  
$ 3 F_{n+1}^2=  5 F_{2n}+2 F_{2n-1}+3 F_{2n+1}-2F_{n-1}^2 -5F_{n}(F_{n-1}+ F_{n+2})$,

\item  \label{FibonaIdentitySat:OSC7}
\begin{multline*}    
F_{n-1} \left(F_{2n}^2-F_{n} F_{2n} F_{n+1}+F_{2n-1}^2-F_{n}^2 (F_{2n-1}+F_{2n+1})\right)   =  \\  2 F_{n-1}^2 F_{n} F_{2n}+F_{n-1}^3 F_{2n-1}  +F_{n}
\left(F_{n}^2 F_{2n}+F_{n} F_{n+1} F_{2n+1}-F_{2n} (F_{2n-1}+F_{2n+1})\right),
\end{multline*} 

\item  \label{FibonaIdentitySat:OSC8}  
$F_{r}((F_{n-1}-F_{n+1})F_{m+2}+(F_{m-1}-F_{m+1})F_{n+2})=((2F_{n}+F_{n+1})F_{m}+F_{m+1}F_{n})(F_{r-1}-F_{r+1})$,

\item  \label{FibonaIdentitySat:OSC9} 
\begin{multline*} 
F_{m-1} ((3F_{n+1}-2 F_{n-1}) F_{r}+F_{ n} (F_{r-1}-F_{r+1}))=\\
F_{m+1} (-3F_{n-1} F_{r}+4 F_{n+1} F_{r}+2 F_{ n} (F_{r-1}-F_{r+1}))-F_{m} (F_{n-1}-2 F_{n+1}) (F_{r-1}-F_{r+1}).
\end{multline*} 

\end{enumerate}
\end{proposition}
 
 \begin{proof}  This proof is a straightforward application of Lemma \ref{JordanIdentities}. In this lemma we use Parts \eqref{JordanIdentity:1}--\eqref{JordanIdentity:9} setting $a=  \mathcal{F}_{1}$, from \eqref{FibonacciGenCase1} and \eqref{FibonacciGenCasePowers1}, $b=  \mathcal{T}_{2}$ from \eqref{TMatrices} and  \eqref {TMatricesPower}
 and $c= \mathcal{L}$ from \eqref {FiboLucasMatrices}.
 
The Proof of Part \eqref{FibonaIdentitySat:OSC1} uses Lemma \ref{JordanIdentities} Part \eqref{JordanIdentity:1}.

The Proofs of Parts \eqref{FibonaIdentitySat:OSC2} and \eqref{FibonaIdentitySat:OSC3} use Lemma \ref{JordanIdentities} Part \eqref{JordanIdentity:2}.

The Proofs of Parts \eqref{FibonaIdentitySat:OSC4}--\eqref{FibonaIdentitySat:OSC6}  use Lemma \ref{JordanIdentities} Part \eqref{JordanIdentity:3}.

The Proof of Part \eqref{FibonaIdentitySat:OSC7}  uses Lemma \ref{JordanIdentities} Part \eqref{JordanIdentity:6}.

The Proofs of Parts \eqref{FibonaIdentitySat:OSC8} and \eqref{FibonaIdentitySat:OSC9}  use Lemma \ref{JordanIdentities} Part \eqref{JordanIdentity:9}.
 \end{proof} 
  
\begin{proposition}\label{SantiagoTheorem7} If $F_{n}$ is a Fibonacci number and $L_{n}$ is a Lucas number, then these identities hold

\begin{enumerate}
\item  \label{FibonaIdentitySat:1}  $F_{n+2}^2-F_{n}^2=F_{2 n+2}$,

\item  \label{FibonaIdentitySat:3}  $F_{n-2}^2+2 F_n^2+2 F_{n+3}^2+8 F_n F_{n+2}  =$ $ 8 F_{n+2}^2+4 F_n F_{n-2}+F_{n+1}^2$, 

\item  \label{FibonaIdentitySat:4}  $F_{n-2}^2+F_{2n+1}+6F_{n}F_{n+2}+2F_{n+1}F_{n+3}=$ $ F_{n-1}L_{n+2}+4F_{n+2}^2+F_{n-2}(3F_{n}+F_{n+2})$,

\item  \label{FibonaIdentitySat:5}  $F_{n-2}^2+2F_{n}L_{n}+2F_{n}F_{n+2}+2F_{n+1}^2=F_{n-1}^2+F_{n}^2+2F_{n+2}^2$,

\item  \label{FibonaIdentitySat:6}   
\begin{multline*} 
3F_n^3+2 F_n^2 \left(F_{n+2}+L_n\right)=F_{n-2} F_n \left(2 F_n+3 L_n\right)+\\
F_n \left(-3 F_{n+2} L_n+2 F_{n+2}^2+2 F_{2 n}+F_{2 n-2}-2 F_{2 n+2}\right) \\
+\left(F_{n+2}^2+3 F_{2 n}-L_{2n}\right) L_n+F_{n-2}^2 \left(F_n+L_n\right),
\end{multline*} 

\item  \label{FibonaIdentitySat:7}  
\begin{multline*} 
5F_n \left(2 F_{n^2} F_{n+2}-\left(F_{n^2-2}-2 F_{n^2+2}\right) F_{n+2}+F_{n^2+n-2}-2 F_{n^2+n}-2 F_{n^2+n+2}\right)=\\
F_n^2L_n\left(-F_{n^2-n-2}+F_{n^2-n+2}-3 F_{(n-1) n}\right) +\\
2F_nL_n \left(F_{n^2-2}-2 F_{n^2+2}+F_{n-2} F_{n^2-n-2}+2 F_{n+2} F_{n^2-n+2}\right)\\
+ L_nF_{n-2} \left(3 F_{n^2}-F_{n^2-2}+F_{n^2+2}+F_{n+2} F_{n^2-n-2}-F_{n+2} F_{n^2-n+2}-3 F_{(n-1) n} F_{n+2}\right) \\
+L_n\left(-3 F_{n^2} F_{n+2}+F_{n+2} F_{n^2-2}-F_{n+2} F_{n^2+2}-F_{n^2+n-2}+3 F_{n^2+n}+F_{n^2+n+2}\right) + \\
2 F_n^2 \left(6 F_{n^2}+3 \left(F_{n^2-2}-F_{n^2+2}+F_{n+2} F_{n^2-n+2}\right)-2 F_{(n-1) n} F_{n+2}\right)+\\
-F_n^3 \left(F_{n^2-n-2}-2 F_{n^2-n+2}+2 F_{(n-1) n}\right)+ 
 F_{n-2}F_n^2\left(6 F_{n^2-n-2}+8 F_{(n-1) n}\right)+\\
 5F_n F_{n-2} \left(2 F_{n^2}-F_{n^2-2}+2 F_{n^2+2}+F_{n+2} F_{n^2-n-2}-2 F_{n+2} F_{n^2-n+2}-2 F_{(n-1) n} F_{n+2}\right),
\end{multline*} 

\item  \label{FibonaIdentitySat:8} 
\begin{multline*}  
F_n\left(F_n \left(6 F_{n^2}-2 F_{n^2-2}+4 F_{n^2+2}\right)+F_n^2 \left(2 F_{(n-1) n}-3 F_{n^2-n+2}\right)-3 F_{n^2+n-2}-4 F_{n^2+n}\right)=\\
L_n\left(F_{n^2-n-2} F_{n-2}^2+\left(2 F_{n^2-2}-3 F_n F_{(n-1) n}\right) F_{n-2}-3 F_n F_{n^2}+2 F_n^2 F_{n^2-n+2}+F_{n^2+n-2}\right) +\\
 F_n F_{n-2}\left(8 F_{n^2}+6 F_{n^2-2}+F_n \left(2 F_{n^2-n-2}-4 F_{n^2-n+2}-6 F_{(n-1) n}\right)\right)+\\
  F_n F_{n-2}^2\left(3 F_{n^2-n-2}+4 F_{(n-1) n}\right),
\end{multline*} 

\item  \label{FibonaIdentitySat:9}  $F_{n-2} F_{2 n-1}F_n^2+F_{2 n} F_{2 n-1} F_n+F_{n+2} F_{2 n+1} \left(F_{n+2}^2-F_{2 n+2}\right)=$ $F_n^2F_{n+2} \left(F_{2 n-1}+F_{2 n+1}\right) $,

\item  \label{FibonaIdentitySat:10}  
\begin{multline*}  
F_{2 n} F_{n+2}F_n^2+\left(-2 F_{2 n}^2+F_{2 n-2} F_{2 n-1}+F_{2 n+1} \left(2 F_{n+2}^2-F_{2 n+2}\right)\right) F_n+\\
F_{n-2}^2 \left(F_{2 n} F_{n+2}+2 F_n F_{2 n-1}\right)+F_{2 n} F_{n+2} \left(F_{2 n-2}-F_{2 n+1}\right)=\\ 
F_{n-2} \left(F_{2 n} F_n^2+F_{n+2} \left(F_{2 n-1}+F_{2 n+1}\right) F_n+F_{2 n} \left(F_{n+2}^2-F_{2 n-1}-F_{2 n+2}\right)\right)+
\left(F_{2 n-1}+F_{2 n+1}\right) F_n^3,
\end{multline*}

\item  \label{FibonaIdentitySat:11}  $F_{2 n-1} F_{n-2}^3+F_n \left(F_n F_{n+2}-F_{2 n}\right) F_{2 n+1}=$ $F_{n-2}\left(F_n^2 \left(F_{2 n-1}+F_{2 n+1}\right)-F_{2 n-2} F_{2 n-1}\right) $,

\item  \label{FibonaIdentitySat:12}  $F_n\left(F_{l-2}+F_{l+2}\right) =F_l \left(F_{n-2}+F_{n+2}\right)$,

\item  \label{FibonaIdentitySat:13} 
\begin{multline*}
F_{2 n+1} F_{l+n+2}^2-F_{2 n-1} F_{l+n}^2=
F_l^2 \left(F_n^2 F_{2 n+1}-F_{n-2}^2 F_{2 n-1}\right)+\\
2 F_{l+2} F_l F_n \left(F_{n-2} F_{2 n-1}-F_{n+2} F_{2 n+1}\right)+F_{l+2}^2 \left(F_{n+2}^2 F_{2 n+1}-F_n^2 F_{2 n-1}\right),
\end{multline*}

\item  \label{FibonaIdentitySat:14} 
\begin{multline*}
F_{2 n+1} F_{l+n} F_{l+n+2}+F_{2 n} F_{l+n}^2=
F_l^2  (F_{2 n} F_n^2+ (F_{n-2} F_{2 n-1}-F_{n+2} F_{2 n+1} ) F_n-F_{n-2} F_{2 n} F_{n+2} )+ \\
F_l  (F_{l-2}  (F_n^2 F_{2 n+1}-F_{n-2}^2 F_{2 n-1} )+F_{l+2}  (F_{n+2}^2 F_{2 n+1}-F_n^2 F_{2 n-1} ) )+\\
F_{2 n} F_{l+n-2} F_{l+n+2}+F_{2 n-1} F_{l+n-2} F_{l+n} +\\
F_{l-2} F_{l+2}  (F_{n-2} F_{2 n-1} F_n-F_{n+2} F_{2 n+1} F_n-F_n^2 F_{2 n}+F_{n-2} F_{2 n} F_{n+2} ),
\end{multline*}

\item  \label{FibonaIdentitySat:15}  
\begin{multline*}
F_{2 n-1} F_{l+n-2}^2=
F_{l-2}^2 \left(F_{n-2}^2 F_{2 n-1}-F_n^2 F_{2 n+1}\right)+\\
2 F_l F_{l-2} F_n \left(F_{n+2} F_{2 n+1}-F_{n-2} F_{2 n-1}\right)+F_{2 n+1} F_{l+n}^2+F_l^2 \left(F_n^2 F_{2 n-1}-F_{n+2}^2 F_{2 n+1}\right),
\end{multline*}

\item  \label{FibonaIdentitySat:16}  
\begin{multline*}
2 F_{n-2}^2 F_n+\left(4 F_{n+2}^2-3 F_{2 n}-6 F_{2 n+2}\right) F_n+2 F_{3 n}+7 F_{n+2} F_{2 n+2}=\\
 4 F_n^3+2 F_{n+2}^3+F_{n-2} \left(F_n^2+2 F_{n+2} F_n-3F_{n+2}^2-4 F_{2 n}+3 F_{2 n+2}\right)+5 F_{3 n+2},
\end{multline*}

\item  \label{FibonaIdentitySat:17}  
\begin{multline*}
24 F_n^2 F_{n+2}+21 F_{n+2} F_{2 n+2}+14 F_{2 n} F_{n+2}+6 F_{3 n-2}  +5 F_{n-2}^3=\\
F_{n-2}^2\left(2 F_n+F_{n+2}\right) +\left(19 F_n^2-13 F_{n+2} F_n-F_{n+2}^2+12 F_{2 n}-11 F_{2 n-2}+F_{2 n+2}\right) F_{n-2}+\\
7 F_n^3+10 F_{n+2}^3+13 F_{3 n}+F_{n+2} F_{2 n-2}+F_n \left(4 F_{n+2}^2+30 F_{2 n}+3 F_{2 n-2}-3 F_{2 n+2}\right)+11 F_{3 n+2}, 
\end{multline*}

\item  \label{FibonaIdentitySat:18} 
\begin{multline*}
6 F_{2 n} F_{n+2}+3 F_{3 n+2}+F_n^3+2F_{n+2}^3  +F_{n+2} F_{2 n-2}=
F_{n-2}^3+\left(2 F_n-F_{n+2}\right) F_{n-2}^2+\\
\left(-3F_n^2-5 F_{n+2} F_n+F_{n+2}^2+4 F_{2 n}+3 F_{2 n-2}-F_{2 n+2}\right) F_{n-2}+\\
5 F_{3 n}+4 F_n^2 F_{n+2}+5 F_{n+2} F_{2 n+2}+F_n \left(4 F_{n+2}^2-6 F_{2 n}+3 F_{2 n-2}-3 F_{2 n+2}\right)+2 F_{3 n-2},
\end{multline*}

\item  \label{FibonaIdentitySat:19}  

\begin{multline*}
 8 F_{2 n} F_{n+2}+3 F_{n+2} F_{2 n-2} +2 F_n^3=
F_{n-2}^3+\left(2 F_n-3 F_{n+2}\right) F_{n-2}^2+\\ 
\left(5 F_{2 n-2}-4 F_n F_{n+2}\right) F_{n-2}+
4 F_{3 n}+2 F_n^2 F_{n+2}+F_n \left(4 F_{n+2}^2-3 F_{2 n}+6 F_{2 n-2}\right)+4 F_{3 n-2},
\end{multline*}

\item  \label{FibonaIdentitySat:O20}  
\begin{multline*}
F_n\left(6 F_{n+2} F_{2 n+2}+2 F_{3 n}-3 F_{n+2}^3-4 F_{2 n} F_{n+2}-3 F_{3 n+2}\right)=\\
L_n\left(F_{n-2} F_n^2-3 F_{n+2} F_n^2+3 F_{2 n} F_n+2 \left(F_{n+2}^3-2 F_{2 n+2} F_{n+2}+F_{3 n+2}\right)\right)+\\
F_n \left(4 F_n^3-6 F_{n+2} F_n^2+F_{n-2} \left(3 F_n+2 F_{n+2}\right) F_n+\left(-6 F_{n+2}^2+6 F_{2 n}-2 F_{2 n-2}+4 F_{2 n+2}\right) F_n\right),
\end{multline*}
 
 \item  \label{FibonaIdentitySat:O21}  
 \begin{multline*}
 F_n (F_n  (6 F_{2 n}-2 F_{2 n-2}+4 F_{2 n+2} )-3 F_{3 n-2} )=\\
 F_n  (3 F_{n-2}^3+6 F_n F_{n-2}^2+ (-6 F_n^2-4 F_{n+2} F_n+8 F_{2 n}+6 F_{2 n-2} ) F_{n-2}-2F_n^3+4 F_{3 n}+3F_n^2 F_{n+2} )+\\
  L_n (F_{n-2}^3+ (2 F_{2 n-2}-3 F_n^2 ) F_{n-2}-3 F_n F_{2 n}+2 F_n^2 F_{n+2}+F_{3 n-2} ).
 \end{multline*}
 
 \end{enumerate}
\end{proposition}  

\begin{proof}  This proof is a straightforward application of Lemma \ref{JordanIdentities}. In this lemma we use Parts \eqref{JordanIdentity:1}--\eqref{JordanIdentity:12}   
setting $a= \mathcal{M}_3$ from \eqref{Combinatorial_McLaughlin3}, 
$b= \mathcal{T}_{2}$ from \eqref {TMatricesPower}, and $c=\mathcal{L}$ from \eqref {FiboLucasMatrices}.

The Proof of Part \eqref{FibonaIdentitySat:1}  uses Lemma \ref{JordanIdentities} Part \eqref{JordanIdentity:1}.

The Proofs of Parts \eqref{FibonaIdentitySat:3}--\eqref{FibonaIdentitySat:5} use Lemma \ref{JordanIdentities} Part \eqref{JordanIdentity:2}.

The Proof of Part \eqref{FibonaIdentitySat:6} uses Lemma \ref{JordanIdentities} Part \eqref{JordanIdentity:3}.

The Proofs of Parts \eqref{FibonaIdentitySat:7} and \eqref{FibonaIdentitySat:8} use Lemma \ref{JordanIdentities} Part \eqref{JordanIdentity:4}.

The Proofs of Parts \eqref{FibonaIdentitySat:9}--\eqref{FibonaIdentitySat:11} use Lemma \ref{JordanIdentities} Part \eqref{JordanIdentity:6}.

The Proof of Part \eqref{FibonaIdentitySat:12} uses Lemma \ref{JordanIdentities} Part \eqref{JordanIdentity:9}.

The Proofs of Parts \eqref{FibonaIdentitySat:13}--\eqref{FibonaIdentitySat:15} use Lemma \ref{JordanIdentities} Part \eqref{JordanIdentity:10}.

The Proofs of Parts \eqref{FibonaIdentitySat:16}--\eqref{FibonaIdentitySat:19} use Lemma \ref{JordanIdentities} Part \eqref{JordanIdentity:12}.

The Proofs of Parts \eqref{FibonaIdentitySat:O20} and \eqref{FibonaIdentitySat:O21} use Lemma \ref{JordanIdentities} Part \eqref{JordanIdentity:13}.
  \end{proof}
  
\begin{proposition}\label{SantiagoLemma2} If $F_{n}$ is a Fibonacci number and $L_{n}$ is a Lucas number, then these identities hold

\begin{enumerate}
\item  \label{FibonaIdentitySat:O1} $8 F_{n-2}+16 F_{n-1}+5 F_{n+2}=5 F_n+13 F_{n+1}$,
\item  \label{FibonaIdentitySat:O2} $F_{n-1}+16 F_n+8 F_{n+1}=11 F_{n+2}+2 F_{n-2} $,
\item  \label{FibonaIdentitySat:O3} $4 F_n+8 F_{n+1}=11 F_{n-1}+6 F_{n-2}+3 F_{n+2}$, 
\item  \label{FibonaIdentitySat:O4}  
 \begin{multline*}
F_{m+1} F_{m+n}+7 F_{m+1} F_{m+n+1}=3 F_{m-1}^2 F_n+2 F_{m+1}^2 F_{n+1}+F_m^2 \left(7 F_n+4 F_{n+1}\right)+\\
F_{m-1} \left(F_{m+1} \left(F_n+3 F_{n+1}\right)+F_m \left(7 F_n+3 F_{n+1}\right)-7 F_{m+n}-3 F_{m+n+1}\right)+\\
F_m \left(2 F_{m+1} \left(F_n+4 F_{n+1}\right)-9 F_{m+n}-11 F_{m+n+1}\right)+4 F_{2 m+n}+5 F_{2 m+n+1},
 \end{multline*}

\item  \label{FibonaIdentitySat:O5} 
 \begin{multline*}
 F_{m-1}\left(11 F_{m+n-1}+12 F_{m+n}+F_{m+n+1}\right)=
 F_{m-1}^2 \left(5 F_{n-1}+F_n\right)+\\
 F_{m-1} \left(F_{m+1} \left(F_{n-1}+10 F_n+F_{n+1}\right)+F_m \left(11 F_{n-1}+12 F_n+F_{n+1}\right)\right)+\\ 
 2 F_{m+1}^2 F_n+10 F_{m+1}^2 F_{n+1}+F_m^2 \left(11 F_{n-1}+13 F_n+6 F_{n+1}\right)  -F_{m+1} F_{m+n-1}+\\ 
  -14 F_{m+1} F_{m+n}-21 F_{m+1} F_{m+n+1}+ 6 F_{2 m+n-1}+13 F_{2 m+n}+11 F_{2 m+n+1} +\\
  F_m \left(2 F_{m+1} \left(F_{n-1}+11 F_n+6 F_{n+1}\right)-13 F_{m+n-1}-34 F_{m+n}-13 F_{m+n+1}\right),
 \end{multline*}

\item  \label{FibonaIdentitySat:O6} 
 \begin{multline*}
  F_{m-1} \left( 4 F_{m+n}+F_{m+n+1}\right)+F_{m+1} F_{m+n-1}+6 F_{m+1} F_{m+n}+5 F_{m+1} F_{m+n+1}= F_{m-1}^2 \left(F_{n-1}+F_n\right)+\\
 F_{m-1} \left(F_{m+1} \left(F_{n-1}+2 F_n+F_{n+1}\right)+F_m \left(3 F_{n-1}+4 F_n+F_{n+1}\right)-3 F_{m+n-1}\right)+ \\
 2 F_{m+1}^2 F_n+2 F_{m+1}^2 F_{n+1}+F_m^2 \left(3 F_{n-1}+5 F_n+2 F_{n+1}\right)+\\
F_m \left(2 F_{m+1} \left(F_{n-1}+3 F_n+2 F_{n+1}\right)-5 \left(F_{m+n-1}+2 F_{m+n}+F_{m+n+1}\right)\right)+\\
2 F_{2 m+n-1}+5 F_{2 m+n}+3 F_{2 m+n+1}.
  \end{multline*}
  
\end{enumerate}
\end{proposition}  

\begin{proof}  This proof is a straightforward application of Lemma \ref{JordanIdentitiesLemma2}. Set $a=\mathcal{F}_{1}^{n}$ from  \eqref{FibonacciGenCasePowers1},    
$b=\mathcal{T}^{n}_{2}$ from  \eqref{TMatricesPower}, and $c= \mathcal{L}^n$ from \eqref{FiboLucasMatrices}.

The Proofs of Parts \eqref{FibonaIdentitySat:O1}--\eqref{FibonaIdentitySat:O3} use Lemma \ref{JordanIdentitiesLemma2} Part \eqref{JordanIdentityL2Part:1}.

The Proofs of Parts \eqref{FibonaIdentitySat:O4}--\eqref{FibonaIdentitySat:O6} use Lemma \ref{JordanIdentitiesLemma2} Part \eqref{JordanIdentityL2Part:2}.
 \end{proof}
 
\subsection{Binomial transform of Fibonacci numbers identities}\label{Identities_from_Jordan_GeneralCaseMcLaughlinThm}

In this section, we use the sequence give in Section \ref{GeneralCaseMcLaughlinThm} and the identities from  
Section \ref{JordanAlgebras}.

\begin{proposition}\label{kthBinomialTransformIden} If $k, n \ge 1$ and $i\in \{-1,0,1\}$ and 
$$h_{n,k}(j)=\sum _{i=0}^n (-1)^{i-1+j}  \binom{n}{i} F_{i-j} (k+1)^{n-i},$$
then these identities hold

\begin{enumerate}
\item \label{kthBinomialTransformIden:S1} $$h_{n,k}(1) (h_{m,k}(1) h_{n,k}(1) - h_{m+n,k}(1)) + h_{n,k}^{2}(0) h_{m,k}(-1)=h_{n,k}(0) (h_{m+n,k}(0) - 2 h_{m,k}(0) h_{n,k}(1)), $$

\item \label{kthBinomialTransformIden:S2}  
\begin{multline*}
h_{m+n,k}(0) (h_{n,k}(1) + h_{n,k}(-1))=
2 h_{m,k}(0) (h_{n,k}^{2}(0) + h_{n,k}(1) h_{n,k}(-1)) + \\   
h_{n,k}(0) (2 h_{m,k}(1) h_{n,k}(1) - h_{m+n,k}(1) + 2 h_{m,k}(-1) h_{n,k}(-1) - h_{m+n,k}(-1)),    
   \end{multline*}   
    
\item \label{kthBinomialTransformIden:S3} 
\begin{multline*}
h_{n,k}^{2}(0) h_{m,k}(1)  +  h_{n,k}(-1) (h_{m,k}(-1) h_{n,k}(-1) - h_{m+n,k}(-1))= \\ 
h_{n,k}(0) (h_{m+n,k}(0) - 2 h_{m,k}(0) h_{n,k}(-1)),
\end{multline*}  
 
 \item \label{kthBinomialTransformIden:S4} 
 \begin{multline*}
 h_{n+1,k}^{2}(0) + h_{n+1,k}^{2}(1)= h_{n,k}^{2}(0) (1+(1+k)^2) + (1+k)^2 h_{n,k}^{2}(1)  + \\ 
 h_{n,k}^{2}(-1) +  2(1+k)  h_{n,k}(0)(h_{n,k}(1) + h_{n,k}(-1)),
\end{multline*} 
  
 \item \label{kthBinomialTransformIden:S5} 
 $h_{n,k}^{3}(1) = h_{n,k}(0) h_{2n,k}(0) + h_{n,k}(1) h_{2n,k}(1) - h_{n,k}^{2}(0) (2 h_{n,k}(1) + h_{n,k}(-1))$,

 \item \label{kthBinomialTransformIden:S6} 
  \begin{multline*}
 2 h_{n,k}^{3}(0)=  h_{2n,k}(0) (h_{n,k}(1) + h_{n,k}(-1)) + h_{n,k}(0) (-2 h_{n,k}^{2}(1) +\\  h_{2n,k}(1) - 2 h_{n,k}(1) h_{n,k}(-1) - 2 h_{n,k}^{2}(-1) +  h_{2n,k}(-1)),
\end{multline*}  
    
\item \label{kthBinomialTransformIden:S7} 
 \begin{multline*}
 h_{n,k}^{2}(1) h_{2n,k}(1)= h_{n,k}^{4}(0) + h_{n,k}^{4}(1) -  h_{n,k}(0) h_{2n,k}(0) (h_{n,k}(1) +  h_{n,k}(-1)) \\+ h_{n,k}^{2}(0) (3 h_{n,k}^{2}(1)- h_{2n,k}(1) + 2 h_{n,k}(1) h_{n,k}(-1) + h_{n,k}^{2}(-1)),
 \end{multline*} 
 
 \item \label{kthBinomialTransformIden:S8} 
  \begin{multline*}
 2 h_{n,k}^{2}(0) h_{2n,k}(0)= 4 h_{n,k}^{3}(0) (h_{n,k}(1) + h_{n,k}(-1)) -  h_{2n,k}(0) (h_{n,k}^{2}(1) + h_{n,k}^{2}(-1)) \\
 +  h_{n,k}(0) (h_{n,k}(1) + h_{n,k}(-1)) (2 h_{n,k}^{2}(1) - h_{2n,k}(1) + 2 h_{n,k}^{2}(-1) - h_{2n,k}(-1)),
 \end{multline*} 

\item \label{kthBinomialTransformIden:S9}  $ h_{n,k}^{3}(1)  + h_{3n,k}(1) +  h_{n,k}^{2}(0) (2 h_{n,k}(1) + h_{n,k}(-1))=2 (h_{n,k}(0) h_{2n,k}(0) + h_{n,k}(1) h_{2n,k}(1))$,

 \item \label{kthBinomialTransformIden:S10} 
 \begin{multline*} 
 h_{2n,k}(0) (h_{n,k}(1) + h_{n,k}(-1))= h_{n,k}^{3}(0) + h_{3n,k}(0)  + h_{n,k}(0) (h_{n,k}^{2}(1) - h_{2n,k}(1) + \\
 h_{n,k}(1) h_{n,k}(-1) + h_{n,k}^{2}(-1) - h_{2n,k}(-1)),
  \end{multline*}

 \item \label{kthBinomialTransformIden:S11} 
 \begin{multline*}  
h_{n,k}(1) ( 2 h_{n,k}(1) h_{2n,k}(1)-h_{n,k}^{3}(1) - h_{3n,k}(1))=
 h_{n,k}^{4}(0)  +\\
   h_{n,k}(0)(h_{3n,k}(0) - h_{2n,k}(0) (3 h_{n,k}(1) + h_{n,k}(-1))) + h_{n,k}^{2}(0) (3 h_{n,k}^{2}(1) - h_{2n,k}(1) \\
   + 2 h_{n,k}(1) h_{n,k}(-1) + h_{n,k}^{2}(-1) -  h_{2n,k}(-1)). 
\end{multline*} 
    
\end{enumerate}
\end{proposition}  

\begin{proof}  
Proof of Parts \eqref{kthBinomialTransformIden:S1}--\eqref{kthBinomialTransformIden:S3}. These proofs are straightforward applications of Lemma \ref{JordanIdentities} Part \eqref{JordanIdentity:1} by setting $a=\mathcal{T}_{k+1}$, and $b=\mathcal{T}_{k+1}$.

Proof of Part \eqref{kthBinomialTransformIden:S4}. This proof is a straightforward application of Lemma \ref{JordanIdentities} 
Part \eqref{JordanIdentity:2} by setting $a=\mathcal{T}_{k+1}$ and $b=I_2$. 

Proof of Parts \eqref{kthBinomialTransformIden:S5} and \eqref{kthBinomialTransformIden:S6}. These proofs are straightforward applications of 
Lemma \ref{JordanIdentities}  Part \eqref{JordanIdentity:3} by setting $a=\mathcal{T}_{k+1}$, and $b=I_2$ and $c=\mathcal{T}_{k+1}$. 

Proof of Parts \eqref{kthBinomialTransformIden:S7} and \eqref{kthBinomialTransformIden:S8}. These proofs are straightforward applications of 
Lemma \ref{JordanIdentities}  Part \eqref{JordanIdentity:6} by setting $a=\mathcal{T}_{k+1}$, and $b=\mathcal{T}_{k+1}$. 

Proof of Parts \eqref{kthBinomialTransformIden:S9} and \eqref{kthBinomialTransformIden:S10}. These proofs are straightforward applications of 
Lemma \ref{JordanIdentities}  Part \eqref{JordanIdentity:12} by setting $a=\mathcal{T}_{k+1}$, and $b=c=I_2$. 

Proof of Part \eqref{kthBinomialTransformIden:S11}. This proof is a straightforward application of 
Lemma \ref{JordanIdentities}  Part \eqref{JordanIdentity:13} by setting $a=\mathcal{T}_{k+1}$, $b=I_2$, and $c=\mathcal{T}_{k+1}$. 
 \end{proof}
 
\subsection{Pell identities from Jordan identities}\label{Pell_Identities_from_Jordan_Identities}

We recall that the  \emph{Pell numbers} sequence is given by the recursive relation   
$p_{n}=2 p_{n-1}+p_{n-2}$, where $p_{0}=0$, $p_{1}=1$. 

\begin{proposition}\label{PellIdentities} If $P_{n}$ is a Pell number, then these identities hold

\begin{enumerate}
\item  \label{PellIdentity:1}  $P_{m} P_{n}+P_{m+1} P_{n+1}$ =$ P_{m+n+1}$, 
\item  \label{PellIdentity:2}  \begin{multline*}
P_{m-1} (n P_{n-1}-n P_{n}+P_{n})+P_{m} ((n+1) P_{n-1}+ 2 n P_{n}-(n-1) P_{n+1})+ \\
P_{m+1} ((n+1) P_{n}+n P_{n+1})=$ $ n (P_{m+n-1}+P_{m+n+1})+2 P_{m+n},
\end{multline*} 
\item  \label{PellIdentity:3}  $P_{n+2}  =P_{n}+2 P_{n+1} $,  
\item  \label{PellIdentity:4}  $ 2 P_{n+2} P_{n+1}+P_{n} P_{n+2} =P_{n}^2+ P_{n+1}(2 P_{n}+ 5 P_{n+1}-P_{n-1})$,
\item  \label{PellIdentity:5}  $ P_{n+1}^2+2 P_{n+2} P_{n+1}-P_{n} P_{n+2}=-P_{n}^2+2 P_{n+1} P_{n}+P_{n+1} (P_{n-1}+4 P_{n+1}) $, 
\item  \label{PellIdentity:6}  $ P_{2 n+1}=  P_{n}^2+P_{n+1}^2 $, 
\item  \label{PellIdentity:7}  $ 4 P_{2 n}+3 P_{2 n-1}+2 P_{2 n+1}=  3 P_{n-1}^2+5 P_{n}^2+2 P_{n+1}^2+4 P_{n} P_{n-1}+4 P_{n} P_{n+1}$,
\item  \label{PellIdentity:8}  $ 7 P_{n}^2=$ $ -3 P_{n-1}P_{n-1}+2 P_{n}P_{n-1}-4P_{n+1}^{2}-2 P_{2 n}+2 P_{n} P_{n+1} +3 P_{2 n-1}+4 P_{2 n+1}$,  
\item  \label{PellIdentity:9} $ 3 P_{n}^2=$ $- P_{n-1}(3 P_{n-1}+2 P_{n})+2 P_{2 n}- 2P_{n} P_{n+1}+3 P_{2 n-1}$, 
\item  \label{PellIdentity:10} 
 \begin{multline*} 2P_{n} (3 P_{m n}-P_{m n+1})+18 P_{n+1} P_{m n+1} = -P_{n}^2 (2 P_{m n-n}+3 P_{m n-n-1})+ \\
 2P_{n+1} (3 P_{m n-n}-P_{m n-n+1}) P_{n}+9 (P_{m n-n+1} P_{n+1}^2+P_{m n+n+1}),
\end{multline*}  
 
\item   \label{PellIdentity:11} 
 \begin{multline*}
P_{n-1} (4 P_{m n}+3 P_{m n-1}+2 P_{m n+1})+P_{n+1} (4 P_{m n}+3 P_{m n-1}+2 P_{m n+1}) \\
+P_{n} (-2 P_{m n}-3 P_{m n-1}+9 P_{m n+1})=  P_{n}^2(2 P_{m n-n}-3 P_{m n-n-1}-4 P_{m n-n+1})+\\ 
P_{n-1} (P_{n+1} (4 P_{m n-n}+3 P_{m n-n-1}+2 P_{m n-n+1})-P_{n} (2 P_{m n-n}+3 P_{m n-n-1}))+ \\
9 P_{n+1} P_{m n-n+1} P_{n}+3 P_{m n+n-1}+4 P_{m n+n}+2 P_{m n+n+1}, 
\end{multline*} 
\item   \label{PellIdentity:12} 
\begin{multline*}
P_{n-1} (2 P_{m n}-3 P_{m n-1}-4 P_{m n+1})+P_{n+1} (2 P_{m n}-3 P_{m n-1}-4 P_{m n+1})+ \\
P_{n} (-2 P_{m n}-3 P_{m n-1}+9 P_{m n+1})$ =  $P_{n}^2 (4 P_{m n-n}+3 P_{m n-n-1}+2 P_{m n-n+1})+\\
9 P_{n+1} P_{m n-n+1} P_{n}-P_{n-1} (P_{n} (2 P_{m n-n}+3 P_{m n-n-1})+ \\ 
P_{n+1} (-2 P_{m n-n}+3 P_{m n-n-1}+4 P_{m n-n+1}))- 3 P_{m n+n-1}+2 P_{m n+n}-4 P_{m n+n+1}, 
\end{multline*}

\item   \label{PellIdentity:13} 
\begin{multline*}
2P_{n} (3 P_{m n}-P_{m n+1})-2P_{n-1} (2 P_{m n}+3 P_{m n-1})+P_{n-1}^2(2 P_{m n-n}+3 P_{m n-n-1})= \\ 
2 P_{n} P_{n-1}(3 P_{m n-n}-P_{m n-n+1}) +9 P_{n}^2 P_{m n- n+1}-3 P_{m n+n-1}-2 P_{m n+n},
\end{multline*}
 
\item   \label{PellIdentity:14} $(n-1) P_{n-1} P_{n}^2=2 P_{n+1} P_{n}^2+(n-1) P_{2 n} P_{n}+(n+1) P_{n+1} \left(P_{n+1}^2-P_{2 n+1}\right)$, 
\item   \label{PellIdentity:15}  
\begin{multline*}
P_{n-1} \left(n P_{n}^2-2 P_{n+1} P_{n}-(n-1) P_{2 n}+n \left(P_{2 n+1}-P_{n+1}^2\right)\right)+P_{n+1} ((n+1) P_{2 n}+n P_{2 n-1})=\\
2 P_{n}^3-n P_{n+1} P_{n}^2+\left(2 (n+1) P_{n+1}^2+2 n P_{2 n}+n P_{2 n-1}-P_{2 n-1}-n P_{2 n+1}-P_{2 n+1}\right) P_{n}+\\
P_{n-1}^2 (n P_{n+1}-2 (n-1) P_{n}), 
\end{multline*}

\item   \label{PellIdentity:16} 
\begin{multline*}
 P_{n-1}  (n P_{n}^2+2 P_{n+1} P_{n}+(n-1) P_{2 n}+n  (P_{2 n+1}-P_{n+1}^2 ) ) +2 P_{n}^3+n P_{n+1} P_{n}^2 =\\ 
   (-2 (n+1) P_{n+1}^2+2 n P_{2 n}-n P_{2 n-1}+P_{2 n-1}+n P_{2 n+1}+P_{2 n+1} ) P_{n} +\\
    P_{n-1}^2 (2 (n-1) P_{n}+n P_{n+1})+ 
   P_{n+1} ((n+1) P_{2 n}-n P_{2 n-1}), 
\end{multline*}

\item   \label{PellIdentity:17} $(n-1) P_{n-1}^3=\left(2 P_{n}^2+(n-1) P_{2 n-1}\right) P_{n-1}+(n+1) P_{n} (P_{n} P_{n+1}-P_{2 n})$, 
\item   \label{PellIdentity:18} $(P_{l-1}-P_{l+1}) (P_{m+1} P_{n}+P_{m} P_{n+1})=P_{l} (P_{m-1} P_{n+1}+P_{m+1} (P_{n-1}-2 P_{n+1}))$,  
\item   \label{PellIdentity:19}  $2 (P_{l-1}-P_{l+1}) P_{m} P_{n}=P_{l} (P_{n}(P_{m-1}-P_{m+1}) +P_{m} (P_{n-1}-P_{n+1}))$,

\item   \label{PellIdentity:20}  
\begin{multline*}
(1-n) P_{l+m}^2 +P_{l+m+1}^2+n P_{l+m+1}^2 =P_{l}^2 \left((n+1) P_{m}^2-(n-1) P_{m-1}^2\right)+\\ 
2 P_{l+1} P_{l} P_{m} ((n+1) P_{m+1}-(n-1) P_{m-1})+P_{l+1}^2 \left((n+1) P_{m+1}^2-(n-1) P_{m}^2\right),
\end{multline*}

\item   \label{PellIdentity:21}  
\begin{multline*}
P_{l+m-1} (n P_{l+m+1}-(n-1) P_{l+m})-  P_{l+m} (n P_{l+m}-(n+1) P_{l+m+1})= \\
 P_{l-1} P_{l+1} (P_{m-1} (-n P_{m}+n P_{m+1}+P_{m})+P_{m} ((n+1) P_{m+1}-n P_{m}))+ \\ 
P_{l} P_{l+1} \left((n+1) P_{m+1}^2-(n-1) P_{m}^2\right)+P_{l}^2 (P_{m} (n P_{m}+(n+1) P_{m+1})+\\
P_{l-1} P_{l} \left((n+1) P_{m}^2-(n-1) P_{m-1}^2\right)+ P_{m-1} (-n P_{m}-n P_{m+1}+P_{m})), 
\end{multline*}

\item   \label{PellIdentity:22}   
\begin{multline*}
P_{l+m+1} ((n+1) P_{l+m}-n P_{l+m-1})-P_{l+m} ((n-1) P_{l+m-1}-n P_{l+m})=\\
P_{l}^2 P_{m-1} (-n P_{m}+n P_{m+1}+P_{m})+P_{l}^2P_{m} ((n+1) P_{m+1}-n P_{m})+ \\
P_{l+1} P_{l} \left((n+1) P_{m+1}^2-(n-1) P_{m}^2\right)+P_{l+1} P_{l-1} P_{m} (n P_{m}+(n+1) P_{m+1})+ \\
P_{l}P_{l-1} \left((n+1) P_{m}^2-(n-1) P_{m-1}^2\right)+ P_{l+1} P_{l-1} P_{m-1} (-n P_{m}-n P_{m+1}+P_{m}), 
\end{multline*}

\item   \label{PellIdentity:23} 
\begin{multline*}
P_{l+m-1}^2+P_{l+m}^2 = P_{l-1}^2 \left((n+1) P_{m}^2-(n-1) P_{m-1}^2\right)+ \\
2 P_{l} P_{l-1} P_{m} ((n+1) P_{m+1}-(n-1) P_{m-1})+n P_{l+m-1}^2-n P_{l+m}^2+\\ 
P_{l}^2 \left((n+1) P_{m+1}^2-(n-1) P_{m}^2\right), 
\end{multline*}

\item   \label{PellIdentity:24} $P_{l+m+1}^2=P_{l}^2 P_{m}^2+2 P_{l} P_{l+1} P_{m+1} P_{m}+P_{l+1}^2 P_{m+1}^2$, 

\item   \label{PellIdentity:25} 
\begin{multline*}
P_{l+m}^2=P_{l+m-1} P_{l+m+1}+2 P_{l+m} P_{l+m+1}-2 P_{l+1} P_{l} P_{m+1}^2- \\
P_{l}^2  (P_{m}^2+ 2 P_{m+1} P_{m}-P_{m-1} P_{m+1} )-P_{l-1}  (2 P_{l} P_{m}^2+ \\
P_{l+1}  (-P_{m}^2+2 P_{m+1} P_{m}+P_{m-1} P_{m+1} ) ), 
\end{multline*}

\item  \label{PellIdentity:26}  
\begin{multline*}
 P_{n} P_{l+m}+P_{m} P_{l+n} +3 P_{m} P_{l+n+1}+P_{m+1} P_{l+n+1}=P_{l} P_{m+1} P_{n}+P_{m-1} P_{l+n+1} \\
 +P_{l} P_{m}(6  P_{n}+ P_{n+1})+(P_{n-1}-3 P_{n}-P_{n+1}) P_{l+m+1}+P_{l+1} (2 P_{m} P_{n}-P_{m-1} P_{n+1}+ \\
 3 P_{m} P_{n+1}+P_{m+1} (-P_{n-1}+3 P_{n}+2 P_{n+1})), 
 \end{multline*}
 
\item  \label{PellIdentity:27} 
\begin{multline*}
P_{n-1} P_{l+m-1} +P_{n-1} P_{l+m+1}+2 P_{l+1} P_{m} P_{n}+20 P_{n} P_{l+m}+2 P_{n} P_{l+m+1}+2 P_{n} P_{l+m-1}+\\ 
P_{n+1} P_{l+m-1}+21 P_{n+1} P_{l+m+1} +P_{m-1} P_{l+n-1}+2 P_{m} P_{l+n-1}+P_{m+1} P_{l+n-1} +20 P_{m} P_{l+n} + \\
P_{m-1} P_{l+n+1}+2 P_{m} P_{l+n+1}+21 P_{m+1} P_{l+n+1}+16 P_{l+m+n}= P_{l+1} P_{m+1} P_{n-1}+10 P_{n-1} P_{l+m}+\\
2 P_{l+1} P_{m+1} P_{n}+P_{l+1} P_{m-1} P_{n+1}+2 P_{l+1} P_{m} P_{n+1}+20 P_{l+1} P_{m+1} P_{n+1}+6 P_{n+1} P_{l+m}+\\
P_{l-1} (P_{m+1} (P_{n-1}+2 P_{n})+P_{m-1} P_{n+1}+2 P_{m} (9 P_{n}+P_{n+1}))+ \\ 
P_{l} (-10 P_{m-1} P_{n+1}+ 4 P_{m} (6 P_{n}+ 5 P_{n+1})+P_{m+1} (4 (5 P_{n}+P_{n+1}) -10 P_{n-1}))+ \\
10 P_{m-1} P_{l+n}+6 P_{m+1} P_{l+n}+2 P_{l+m+n-1}+22 P_{l+m+n+1}, 
 \end{multline*}
\item  \label{PellIdentity:28} 
\begin{multline*}
6 P_{n+1} P_{l+m}+3 P_{m} P_{l+n-1}+P_{m} P_{l+n}+6 P_{m+1} P_{l+n}+2 P_{l+m+n-1}+P_{n} P_{l+m}= P_{l} P_{m+1} P_{n}+\\
P_{l-1} (4 P_{m} P_{n}-P_{m+1} (P_{n-1}-3 P_{n})-P_{m-1} P_{n+1}+3 P_{m} P_{n+1})+P_{m-1} P_{l+n-1}+\\ 
P_{m+1} P_{l+n-1}+ P_{l} P_{m} P_{n+1}+6 P_{l} P_{m+1} P_{n+1}+(P_{n-1}-3 P_{n}+P_{n+1}) P_{l+m-1}+6 P_{l+m+n}, 
 \end{multline*}

\item  \label{PellIdentity:29} 
\begin{multline*} 
P_{n} (P_{l+1} (P_{m}-3 P_{m+1})-P_{l+m}+3 P_{l+m+1})=\\
P_{l} (-2 P_{m-1} P_{n}+6 P_{m} P_{n}-P_{m} P_{n+1}+3 P_{m+1} P_{n+1}+P_{m+n}-3 P_{m+n+1}),
 \end{multline*}  

\item  \label{PellIdentity:30}  
\begin{multline*} 
P_{n-1} P_{l+m-1}+11 P_{l-1} P_{m+n+1}+16 P_{l+m+n}+P_{n+1} P_{l+m-1}+P_{l-1} P_{m+n-1}= \\
4 P_{n} P_{l+m-1}-12 P_{n} P_{l+m}+P_{l-1} P_{m-1} P_{n+1}-8 P_{l-1} P_{m} P_{n+1}+ \\
11 P_{l-1} P_{m+1} P_{n+1}+8 P_{n+1} P_{l+m}-11 P_{n+1} P_{l+m+1}+8 P_{l-1} P_{m+n} +\\
-  2 P_{l} (-6 P_{m} P_{n-1}-4 P_{m} P_{n}-P_{m+1} P_{n}+P_{m-1} (2 P_{n-1}+P_{n})-2 P_{m+n-1}+6 P_{m+n})+\\ 
P_{l+1} (P_{m-1} P_{n-1}-8 P_{m} P_{n-1}+11 P_{m+1} P_{n-1}-P_{m+n-1}+8 P_{m+n}-11 P_{m+n+1})+ \\
8 P_{n-1} P_{l+m}-11 P_{n-1} P_{l+m+1}-4 P_{l-1} P_{m-1} P_{n}+12 P_{l-1} P_{m} P_{n}+ 
2 P_{l+m+n-1}+22 P_{l+m+n+1},
 \end{multline*}
 
\item  \label{PellIdentity:31}  
\begin{multline*} 
6 P_{n-1} P_{l+m}+3 P_{n} P_{l+m+1}+2 P_{l+m+n-1}=2 P_{n-1} P_{l+m-1}+6 P_{l+m+n}+P_{n} P_{l+m}+\\
P_{l-1} (-2 P_{m-1} P_{n-1}+6 P_{m} P_{n-1}-P_{m} P_{n}+3 P_{m+1} P_{n}+2 P_{m+n-1}-6 P_{m+n})+ \\
P_{l} (-P_{m} P_{n-1}+3 P_{m+1} P_{n-1}+P_{m+n}-3 P_{m+n+1}).
 \end{multline*} 
\end{enumerate}
\end{proposition}  

\begin{proof} This proof is a straightforward application of Lemma \ref{JordanIdentities}. In this lemma we use Parts  
\eqref{JordanIdentity:1}--\eqref{JordanIdentity:12} setting $a=\mathcal{P}_{2}$ from  \eqref{FibonacciGenCasePowers2}, 
$b= \mathcal{M}_1$ from \eqref {Combinatorial_McLaughlin1}, $c=\mathcal{M}_2$ from \eqref {Combinatorial_McLaughlin2},  
$d=\mathcal{S}_{k}$ from \eqref{SMatrices}, and to use Parts \eqref{JordanIdentity:12} and \eqref{JordanIdentity:13} of the lemma  
we set $c=\mathcal{L}$ from \ref{FiboLucasMatrices}.

The Proofs of Parts \eqref{PellIdentity:1} and  \eqref{PellIdentity:2} use Lemma \ref{JordanIdentities} Part \eqref{JordanIdentity:1}.

The Proofs of Parts \eqref{PellIdentity:3}--\eqref{PellIdentity:5} use Lemma \ref{JordanIdentities} Part \eqref{JordanIdentity:2}.

The Proofs of Parts \eqref{PellIdentity:6}--\eqref{PellIdentity:9} use Lemma \ref{JordanIdentities} Part \eqref{JordanIdentity:3}.

The Proofs of Parts \eqref{PellIdentity:10}--\eqref{PellIdentity:13} use Lemma \ref{JordanIdentities} Part \eqref{JordanIdentity:4}.

The Proofs of Parts \eqref{PellIdentity:14}--\eqref{PellIdentity:17} use Lemma \ref{JordanIdentities} Part \eqref{JordanIdentity:6}.

The Proofs of Parts \eqref{PellIdentity:18} and \eqref{PellIdentity:19} use Lemma \ref{JordanIdentities} Part \eqref{JordanIdentity:9}.

The Proofs of Parts \eqref{PellIdentity:20}--\eqref{PellIdentity:23} use Lemma \ref{JordanIdentities} Part \eqref{JordanIdentity:10}.

The Proofs of Parts \eqref{PellIdentity:24}  and \eqref{PellIdentity:25} use Lemma \ref{JordanIdentities} Part \eqref{JordanIdentity:2}.

The Proofs of Parts \eqref{PellIdentity:26}--\eqref{PellIdentity:28} use Lemma \ref{JordanIdentities} Part \eqref{JordanIdentity:12}.

The Proofs of Parts \eqref{PellIdentity:29}--\eqref{PellIdentity:31} use Lemma \ref{JordanIdentities} Part \eqref{JordanIdentity:13}.
\end{proof}

\section{Appendix. Mathematica programing}\label{Appendix1} 
In this section, we share our programs that we made in Mathematica. Where $\text{Mc}[A\_, n\_]$ is $A^n$ given in Theorem \ref{McLaughlinThm}, 

\textbf{Input}.  An integer $n$ and a matrix 

$A=
\left[\begin{array}{ll}
a& b \\
c & d \\
\end{array}
\right]$. 

\textbf{Output}. Matrix with sequences associated to $A$.

\subsection{Construction of Mc Laughlin Matrix from Theorem \ref{McLaughlinThm}}
Here $\textsc{Y}[A\_,n\_]$ is $y_n$ and $\textsc{Mc}[A\_, n\_]$ is $A^n$ as given in Theorem \ref{McLaughlinThm}. 

$\textsc{Y}[A\_,n\_]:=\sum _{i=0}^{\textsc{Floor}[ \frac{n}{2}]}\textsc{Binomial}[n-i,i] \textsc{Tr}[A]^{n-2 i} (-\textsc{Det}[A] )^i $;

\begin{multline*} \textsc{Mc}[A\_, n\_] := 
 \{\{\textsc{Y}[A, n] - A[[2]][[2]]*\textsc{Y}[A, n - 1],  A[[1]][[2]]*\textsc{Y}[A, n - 1]\}, \\ \{A[[2]][[1]]*\textsc{Y}[A, n - 1],  \textsc{Y}[A, n] - A[[1]][[1]]*\textsc{Y}[A, n - 1]\}\};
 \end{multline*}

\subsection{Construction of Mc Laughlin Matrix using \eqref{McLaughlinBinet} and \eqref{McLaughlinSeq}}
Here  $\textsc{Z}[A\_,n\_]$ is as in \eqref{McLaughlinBinet}  and  $\textsc{McIden}[A\_, n\_]$ is as in $\eqref{McLaughlinSeq}$.

$\alpha[A\_] = (1/2) (\textsc{Tr}[A] + \sqrt{\textsc{Tr}[A]^2 - 4 \textsc{Det}[A]});$

$\beta[A\_] =(1/2) (\textsc{Tr}[A] - \sqrt{\textsc{Tr}[A]^2 - 4 \textsc{Det}[A]});$

$\textsc{Z}[A\_,n\_]:=\textsc{Simplify}\left[\frac{\alpha [A]^n-\beta [A]^n}{\alpha [A]-\beta [A]}\right];$

\begin{multline*} 
\textsc{McIden}[A\_, n\_] := \{\{\textsc{Z}[A, n]*A[[1]][[1]] - \textsc{Z}[A, n - 1]*\textsc{Det}[A], \textsc{Z}[A, n]*A[[1]][[2]]\}, \\ \{\textsc{Z}[A, n]*A[[2]][[1]], \textsc{Z}[A, n]*A[[2]][[2]] - \textsc{Z}[A, n - 1]*\textsc{Det}[A]\}\};
\end{multline*}

\subsection{Using Jordan Identities}
In this section, we give functions to evaluate the Jordan product and the ternary Jordan product and one of the identities from Section \ref{JordanAlgebras} 
(we give only one identity, in a similar way the other identities can be defined).  
Here $\textsc{JordanP}[A\_, B\_]$ is the Jordan product and $\textsc{TernaryP}[A\_, B\_, C\_]$ is the ternary product. 

\textbf{Input}.  An integer $n$ and $2\times2$ matrices $A$, $B$, $C$. 

\textbf{Output}. An identity of matrices.

$\textsc{JordanP}[A\_, B\_] := (1/2) (A.B + B.A)$;

$\textsc{TernaryP}[A\_, B\_, C\_] := (1/2) ((A.B).C + (C.B).A)$;

\begin{multline*} 
\textsc{Identity1}[An\_, Am\_, Bn\_, AmSn\_] := \textsc{Print}[\textsc{MatrixForm}[TernaryP[An, Am, Bn]] , \\ `` = ",  \textsc{MatrixForm}[\textsc{JordanP}[AmSn, Bn]]];
\end{multline*}

This function can be used with any matrices associated to a recurrence relation.  For example, if $A= \{\{1, 1\}, \{1, 0\}\}$, we can take $An=\textsc{McIden}[A, n]$, 
$Am=\textsc{McIden}[A, m]$, $Bn=\{\{1,0\},\{0,1\}\}$, and $AmSn=\textsc{McIden}[A, m+n]$ into $\textsc{Identity1}[An, Am, Bn, AmSn] $ to obtain fibonacci numerical   
values $n$ and $m$. (If we want a symbolic identity it is possible to do some manipulation on $\textsc{Binomial}[n-i,i]$ such that it provides symbolic results).

Note 1. The coding in Mathematica for the identities and some matrices will be available on the webpage  \url{http://macs.citadel.edu/florez/research.html} .

Note 2. Again, there are still many things, on how this connection works, that we would like to understand better. For example, we are wondering  
under what conditions the identities given by Glennie \cite{Glennie} can be used to obtain new identities --under the context of this paper. 
We only know that some identities associated to powers have good behavior.

\section{Acknowledgement}  	
The last author was partially supported by The Citadel Foundation.

\end{document}